\documentclass[pdflatex,12pt,a4paper,twoside]{article}
\usepackage{authblk}
\usepackage{amsmath,amsfonts,amssymb,amsthm}
\usepackage{ifthen}
\usepackage{graphicx}
\usepackage{epsfig}
\usepackage{nicefrac}
\usepackage{mathrsfs}
\usepackage[a4paper,left=27mm,right=27mm,top=34mm,bottom=34mm]{geometry}
\usepackage{amsmath}
\usepackage{amsfonts}
\newcommand{\R}{\mathbb{R}}
\newcommand{\N}{\mathbb{N}}
\newcommand{\C}{\mathbb{C}}
\newcommand{\Z}{\mathbb{Z}}
\newcommand{\eps}{\varepsilon}

\newcommand{\ra}{\rangle}
\newcommand{\la}{\langle}
\newcommand{\textmean}{-\hspace{-0.85em}\int}
\newcommand{\del}{\partial}

\newcommand{\supp}{\mathrm{supp}}

\newcommand{\loc}{\mathrm{loc}}

\newcommand{\inc}{\mathrm{inc}}
\newcommand{\ver}{\mathrm{vert}}
\newcommand{\ev}{\mathrm{ev}}
\newcommand{\out}{\mathrm{out}}

\def\calF{\mathcal{F}}

\def\calL{\mathcal{L}}
\def\calM{\mathcal{M}}

\DeclareMathOperator*{\argmin}{\arg\!\min} 

\renewcommand{\Re}{\mathrm{Re}}
\renewcommand{\Im}{\mathrm{Im}}

\newcommand*{\medcap}{\mathbin{\scalebox{1.3}{\ensuremath{\cap}}}}

\def\Xint#1{\mathchoice
   {\XXint\displaystyle\textstyle{#1}}%
   {\XXint\textstyle\scriptstyle{#1}}%
   {\XXint\scriptstyle\scriptscriptstyle{#1}}%
   {\XXint\scriptscriptstyle\scriptscriptstyle{#1}}%
   \!\int}
\def\XXint#1#2#3{{\setbox0=\hbox{$#1{#2#3}{\int}$}
     \vcenter{\hbox{$#2#3$}}\kern-.5\wd0}}
\def\meanint{\Xint-}

\newtheorem{theorem}{Theorem}[section]
\newtheorem{definition}[theorem]{Definition}
\newtheorem{lemma}[theorem]{Lemma}
\newtheorem{proposition}[theorem]{Proposition}
\newtheorem{corollary}[theorem]{Corollary}
\newtheorem{assumption}[theorem]{Assumption}

\newtheorem{remark}[theorem]{Remark}

\newtheorem{problem}[theorem]{Problem}

\numberwithin{equation}{section}

\title{\bf\Large Outgoing wave conditions in photonic crystals and transmission
  properties at interfaces}

\author{A.\,Lamacz, B.\,Schweizer\thanks{Technische Universit\"at
    Dortmund, Fakult\"at f\"ur Mathematik, Vogelpothsweg 87, D-44227
    Dortmund, Germany.  \tt agnes.lamacz@tu-dortmund.de,
    ben.schweizer@tu-dortmund.de }}

\date{December 19, 2016}

\begin{document}

\maketitle

\begin{abstract}
  We analyze the propagation of waves in unbounded photonic crystals.
  Waves are described by a Helmholtz equation with $x$-dependent
  coefficients, the scattering problem must be completed with a
  radiation condition at infinity. We develop an outgoing wave
  condition with the help of a Bloch wave expansion. Our radiation
  condition admits a uniqueness result, formulated in terms of the
  Bloch measure of solutions.  We use the new radiation condition to
  analyze the transmission problem where, at fixed frequency, a wave
  hits the interface between free space and a photonic crystal. We
  show that the vertical wave number of the incident wave is a
  conserved quantity. Together with the frequency condition for the
  transmitted wave, this condition leads (for appropriate photonic
  crystals) to the effect of negative refraction at the interface.
\end{abstract}



\pagestyle{myheadings} 
\thispagestyle{plain} 

\markboth{A.\,Lamacz, B.\,Schweizer}{Outgoing wave conditions and
  transmission at interfaces of photonic crystals}

\section{Introduction}

Photonic crystals are optical devices that allow to mold the
propagation properties of light. They usually have a periodic
structure and are operated with light at a fixed frequency $\omega$.
Due to their spectral properties (band gap structure), light of
certain frequencies can travel in the photonic crystal, but, at other
frequencies, the crystal is opaque.  A large body of literature is
available on this aspect of photonic crystals. Most contributions
study a periodic medium, possibly with a compactly supported
perturbation of the periodic structure. In contrast, we are interested
in the interface between a photonic crystal and free space.

An interesting effect of such an interface is negative refraction. A
recent discussion in the physical literature concerns the following
question: Is negative refraction always the result of a negative index
of the photonic crystal, or can negative refraction also occur at the
interface between air and a photonic crystal with positive index? Our
mathematical results confirm the latter: The conservation of the
transversal wave number can lead to negative refraction between two
materials with positive index, as suggested in \cite
{PhysRevB.65.201104}.

In mathematical terms, the light intensity is determined by the
Helmholtz equation
\begin{equation}
  \label{eq:Peps}
  -\nabla\cdot (a(x) \nabla u(x)) = \omega^2\,  u(x)\,,
\end{equation}
which must be solved for $u$ in a domain $\Omega$, $u = u(x)$, $x =
(x_1,x_2) \in \Omega$. Here, we restrict our analysis to an unbounded
rectangle $\Omega: = \R\times (0,h) \subset \R^2$, but note that our
methods can also be used in higher dimension, e.g.\,for $\Omega: =
\R\times (0,h_2)\times (0,h_3) \subset \R^3$. In \eqref {eq:Peps},
$\omega>0$ is a prescribed frequency and $a = a(x)$ is the inverse
permittivity of the medium. In an $x_3$-independent geometry and with
polarized light, the time-harmonic Maxwell's equations reduce to
\eqref {eq:Peps} and $u = u(x)$ is the out-of-plane component of the
magnetic field.

The coefficient $a = a(x)$ describes the medium. We assume that the
right half space $\{ x= (x_1,x_2) \in \Omega | x_1 >0\}$ is occupied
by a periodic photonic crystal with periodicity length $\eps>0$.
Using the unit cube $Y = (0,1)^2$ and the scaled cube $Y_\eps = \eps Y
= (0,\eps)^2$, we therefore assume that the coefficient $a = a^\eps$
is $Y_\eps$-periodic for $x_1>0$. We make the assumption that an
integer number $K$ of cells fits vertically in the domain, i.e.\,$K =
h/\eps \in \N$.  On the left half space $\{ x= (x_1,x_2) \in \Omega |
x_1 <0\}$, we set $a = a^\eps \equiv 1$.  With $a = a^\eps$ and
$\omega$ given, problem \eqref {eq:Peps} is an equation for $u$, but
it must be accompanied by boundary conditions.


We impose periodic boundary conditions in the vertical direction,
i.e.\,we identify the lower boundary $\{ x = (x_1,x_2) | x_2 = 0 \}$
with the upper boundary $\{ x = (x_1,x_2) | x_2 = h \}$. In order to
analyze scattering properties of the interface, we assume that the
interface is lit by a planar wave. We consider, for a fixed
wave-vector $k\in \R^2$, the incident wave
\begin{equation}
  \label{eq:u-inc}
  U_\inc(x) = e^{2\pi i k\cdot x/\eps}\,.
\end{equation}
To guarantee that $U_\inc$ is a solution to \eqref {eq:Peps} on the
left, we assume $\omega^2 = 4\pi^2 |k|^2 / \eps^2$.  Since the
Helmholtz-equation models a time-harmonic situation, we should think
here of a solution of the wave equation in the form $\hat U_\inc(x,t)
= U_\inc(x) e^{-i\omega t} = \exp(i [2\pi k\cdot x/\eps - \omega
t])$. We always consider $k_1 > 0$ such that $U_\inc$ represents a
right-going wave. In addition, we assume that the incident wave
respects the periodicity condition in vertical direction,
i.e.\,$e^{2\pi i k_2 h/\eps} = 1$ or, equivalently, $k_2 h \in \eps
\Z$.

With the incident wave $U_\inc$ at hand we can now describe -- at
least formally -- the boundary conditions for solutions $u$ of \eqref
{eq:Peps} as $x_1\to \pm\infty$. We seek for $u$ such that (i) $u$
satisfies an outgoing wave condition as $x_1\to \infty$ and (ii) $u -
U_\inc$ satisfies an outgoing wave condition for $x_1\to-\infty$. This
leads us to our first question:
\begin{quote}
  {\bf Question 1:} How can we prescribe radiation conditions in
  periodic media?
\end{quote}
The answer to Question 1 is intricate and requires a detailed
study. We will use Bloch expansions and Bloch projections to formulate
our new outgoing wave condition in Definition
\ref{def:outgoingwave}. In order to motivate our choice, we sketch
some background in the next subsection.

\begin{figure}[ht]
  \begin{center}
    \includegraphics[width=0.8\textwidth]{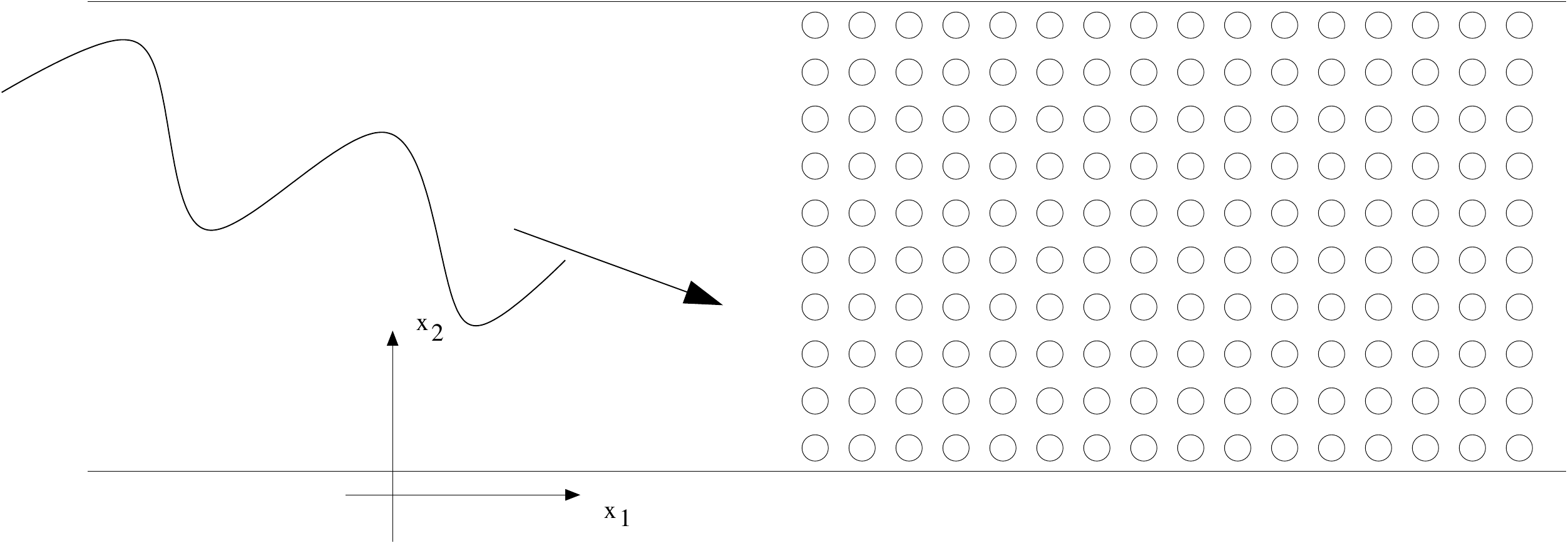}
    \caption{The geometry of the transmission problem for $K = 10$
      (number of cells in vertical direction). An incoming wave hits
      the boundary of a photonic crystal. We are interested in the
      waves that are  generated in the photonic crystal. }
      \label{F:Geometry}
  \end{center}
\end{figure}

\medskip Once we have a precise formulation of the scattering problem,
we can turn to the application: What can be said about the
transmission problem? When an incident wave $U_\inc$ lights the
interface, it creates waves inside the photonic crystal. These waves
are described by $u$ on $\{x_1>0\}$, our aim is to characterize $u$.
Performing a Bloch expansion, we write $u$ as a superposition of Bloch
waves. In this superposition, we expect that there appear only waves
that satisfy two requirements: (a) the Bloch frequency coincides with
the frequency $\omega$. (b) the vertical wave number of the Bloch wave
is $k_2$ (``conservation of the vertical wave number'').
\begin{quote}
  {\bf Question 2:} Let $u$ be the solution of the transmission
  problem for the incoming wave $U_\inc$. Does the Bloch expansion of
  $u$ on the right respect the frequency condition and the
  conservation of the vertical wave number?
\end{quote}

A positive answer to Question 2 provides information on the negative
refraction phenomenon. The requirements (a) and (b) are used in \cite
{PhysRevB.65.201104} to explain negative refraction without referring
to a negative index material: Denoting the $m$-th Bloch eigenvalue for
the wave-vector $j\in Z := [0,1]^2$ as $\mu_m(j)$, the photonic
crystal can have the property that the three conditions (a) $\mu_0(j)
= \omega^2$, (b) $j_2 = k_2$, and the additional condition (c)
$e_1\cdot \nabla_j\mu_0(j) > 0$ (the group velocity has a positive
$x_1$-component), determine $j$ uniquely. For an appropriately chosen
field $a$, an appropriate frequency $\omega$ and an appropriate
incoming wave vector $k$, we have the following situation: $e_2\cdot k
= k_2$ is negative, but the solution $j$ satisfies $e_2\cdot \nabla_j
\mu_0(j) > 0$.  This means that a light beam that hits the interface
from above ($k_2$ negative in free space implies that the vertical
group velocity is negative) produces a light beam in the photonic
crystal that is directed towards the top (vertical group velocity is
positive, $e_2\cdot \nabla_j \mu_0(j) > 0$). With this mechanism, the
conditions (a)--(b) can lead to negative refraction. This is outlined
in \cite {PhysRevB.65.201104}, where a specific photonic crystal is
described and the negative refraction effect is supported by numerical
results.  We note that a quite different interpretation is given in
\cite {EfrosPokrovsky-SolidState-2004}.

\smallskip We will answer the above Questions 1 and 2. The precise
answers are more complex than one might expect at first sight (we
sketch some of the principal difficulties in the next two
subsections). We show that our outgoing wave condition of Definition
\ref{def:outgoingwave} is reasonable by proving a uniqueness result:
Theorem \ref {thm:uniqueness} yields, in a weak sense, the uniqueness
of solutions in terms of the Bloch measure. Question 2 is answered
with Theorem \ref {thm:VertWaveNumber}: If $u$ is a solution that
satisfies the outgoing wave condition, then the corresponding Bloch
measure is concentrated in those frequencies that respect (a)--(b).
The mathematical description of our results is given in Section \ref
{ssec.main}.

\subsection{Outgoing wave conditions}

Although we use our results to analyze negative refraction, the core
of our mathematical theory is more general: We develop an outgoing
wave condition for the Helmholtz equation in a periodic medium.  In
this section, we sketch some background concerning radiation
conditions, mainly in free space.  Our aim is to demonstrate the
importance of radiation conditions, to show the intimate link between
radiation conditions and uniqueness results, and to motivate our
mathematical approach.
 
The Helmholtz equation \eqref {eq:Peps} has been studied already by
Euler and Lagrange, but Helmholtz was the first who expressed
solutions in bounded domains with a representation formula
\cite{Helmholtz-1860}. In unbounded domains, one faces the problem of
boundary conditions at infinity. We recall that two fundamental
solutions of the Helmholtz equation for $x\in\R^3$ are given by
\begin{equation}
  \label{eq:fund-soln-R3}
  u_\out(x) := \frac{1}{|x|} e^{i\omega |x|}\quad\text{ and }\quad
  u_\inc(x) := \frac{1}{|x|} e^{-i\omega |x|}\,.
\end{equation}
With the time-dependence $e^{-i\omega t}$, the solution $u_\out$
represents an outgoing wave, $u_\inc$ an incoming wave. Outgoing waves
are expected to be the building stones of solutions of scattering
problems, incoming waves should not be present in the expansion of
solutions.

Sommerfeld introduced in \cite{Sommerfeld-1912} for dimension $n=3$ a
radiation condition; until today, it is the standard outgoing wave
condition in free space and is named after him:
\begin{equation}
  \label{eq:Sommerfeld}
  |x|^{(n-1)/2} (\del_{|x|} u - i\omega u)(x) \to
  0\quad \text{  as } |x|\to \infty\,.
\end{equation}
The solution $u_\out$ satisfies \eqref {eq:Sommerfeld} and is
therefore admissible, $u_\inc$ does not satisfy \eqref {eq:Sommerfeld}
and is not accepted as a solution.  Sommerfeld justified his radiation
condition with a uniqueness proof: Prescribing boundary data on an
obstacle (the scatterer) and the radiation condition \eqref
{eq:Sommerfeld} at infinity, the Helmholtz equation has at most one
solution. Actually, Sommerfeld demanded two further properties to
guarantee uniqueness, but the results of Rellich (today known as
``Rellich Lemma'') showed that the additional assumptions are not
necessary \cite{Rellich-1943}, see \cite{Schot-Sommerfeld-1992} for
the historical background.

For two reasons, we cannot use the Sommerfeld radiation condition. The
first is that we consider $x$-dependent coefficients $a$. The interest
in $x$-dependent coefficients is not new: Sommerfeld himself studied
the case that $a$ takes two different values in two disjoint
half-spaces, J\"ager studied in \cite{Jaeger-1967} coefficients $a$
that stabilize to constant coefficients for $|x|\to \infty$. Our
situation is different, since $a$ is periodic in the right half plane.
The second reason is that we study a waveguide such that, in the above
sense, our situation is neither one- nor two-dimensional.  For
constant coefficients, the elementary solution in a strip $\R\times
(0,1)$ is $e^{i \kappa\cdot x}$ with $\kappa = (\kappa_1, \kappa_2)$
and $|\kappa|^2 = \omega$, which is right-going for $\kappa_1>0$ and
left-going for $\kappa_1<0$. The solution has no decay (similar to the
one-dimensional case), but the expression $\del_{x_1} u - i\omega u$
does not vanish for right-going waves due to the presence of
$\kappa_2$.

The idea of our outgoing wave condition is simple: Every function on a
rectangular domain can be expanded in Bloch waves. We demand that the
expansion of the solution contains only outgoing waves. The precise
form of the outgoing wave condition is slightly more technical since
we have to consider restrictions of the solution $u$ to large
rectangles (in order to have a small contribution of non-periodicity
effects).

We emphasize that, even though the conditions become more technical,
we follow the historical pathway: The expression in \eqref
{eq:Sommerfeld} can be understood as a projection of the solution $u$
onto incoming waves. Outgoing waves are filtered out and it is
demanded that the remainder is small for large radii.  Our outgoing
wave condition \eqref {eq:outgoingright} is: At the far right, the
solution $u$ can be expanded in a Bloch series that contains only
right-going waves. With this requirement, we follow once more
Sommerfeld who writes in \cite{Sommerfeld-1912}: ``at infinity $u$
must be representable as a sum (or integral) of waves of the divergent
traveling type.'' For an extensive study of homogeneous media we refer
to \cite {ColtonKress}.

\subsubsection*{On radiation in waveguides and  photonic crystals}

An important contribution is the recent work of Fliss and Joly \cite
{FlissJoly2016}, which is also concerned with outgoing wave conditions
and the existence and the uniqueness of solutions for periodic
wave-guides. Essentially, the outgoing wave condition of \cite
{FlissJoly2016} (for $x_1\to +\infty$) reads
\begin{equation}
  \label{eq:FlissJolyCond}
  u(x) = \sum_{\lambda\in N(\omega)} \alpha^+_\lambda U^+_\lambda(x) +
  w^+(x)\,,
\end{equation}
where $N(\omega)$ is a finite index set, $\alpha^+_\lambda$ are real
coefficients, $U^+_\lambda$ are right-going Bloch waves and $w^+(x)$
is exponentially small for $x_1\to +\infty$.  The setting of the
problem differs in one important point from ours: \cite
{FlissJoly2016} studies a medium which is identical at the far left
and at the far right, which allows to use global Floquet-Bloch
transformations; this is not possible in our setting. Below we give a
more detailed comparison of our results to those of \cite
{FlissJoly2016}.

Another radiation condition in a waveguide with varying index in
transversal direction is studied in
\cite{Bonnet-Ben-etal-SIAP2009}. The ``modal radiation condition'',
formulated in Definition 2.4 of \cite{Bonnet-Ben-etal-SIAP2009},
demands for solutions $u$ of the Helmholtz equation that
\begin{equation*}
  (\calF u(x,.))(\lambda) = \hat\alpha^\pm_\lambda e^{-\sqrt{\lambda} |x|}
  \qquad \text{ for } \pm x> a
\end{equation*}
holds for every $\lambda$. Here, $\calF$ denotes a generalized Fourier
transform and $x$ is the longitudinal independent variable. As in our
radiation condition, it is demanded that only outgoing waves
($e^{-\sqrt{\lambda} |x|}$ instead of $e^{+\sqrt{\lambda} |x|}$) are
present.  The new feature of our approach is that it covers media with
oscillations also in longitudinal direction.  We cannot use methods
that rely on separation of variables and the Fourier transform must be
replaced by a Bloch transform.

Also in \cite{Ammari-Antenna-2001}, the radiation of waves inside a
photonic crystal is investigated, and the setting also uses an
interface between a photonic crystal and free space. The fundamental
difference to our work is that in \cite {Ammari-Antenna-2001} the
underlying frequency $\omega$ is assumed to lie in a band-gap of the
photonic crystal. For this reason, waves in \cite
{Ammari-Antenna-2001} are found to decay exponentially in the photonic
crystal and no explicit radiation condition must be formulated.  We
mention \cite {Nazarov2014} and the references therein for other
approaches to radiation conditions, also based on Poynting vectors and
incoming and outgoing waves.

For numerical calculations, one is interested in replacing the
unbounded domain by a bounded domain. In this case, one asks for
appropriate boundary conditions that must be imposed on the boundary
of the bounded domain. This point of view leads to the construction of
Dirichlet-to-Neumann maps or similar ideas \cite{Fliss-2013,
  Joly-Fliss-2012, FlissKlindworthSchmidt2015}. Other key-words are
perfectly matched layers \cite{Joly-PML} or transparent boundary
conditions. These approaches 
may be an alternative to the outgoing wave condition that we
suggest here. However, we are not aware of any result for such 
boundary conditions that implies our uniqueness theorem.

\subsection{Uniqueness and negative refraction}

Following Sommerfeld's example, we accompany our outgoing wave
condition in photonic crystals with a uniqueness statement.  With this
result, we can treat the application on negative refraction. Several
problems must be tackled in this process and the uniqueness result is,
unfortunately, not as strong and simple as one would like it to be.

It is an essential feature of the Helmholtz equation that, even
without source terms and with homogeneous boundary conditions,
solutions may be nontrivial. One example is the bounded domain $\Omega
= (0,1)\subset \R^1$ with the solution $u(x) = \sin(\pi x)$ for
$\omega = \pi$. A more relevant example in higher dimension ($2$ or
$3$) is the Helmholtz resonator: When $\omega$ coincides with the
resonance frequency, there is a nontrivial solution to homogeneous
boundary conditions, see \cite{
  Schweizer-Helmholtz-resonator}.  For regular exterior domains, the
Sommerfeld condition implies uniqueness: The Helmholtz operator has
only a continuous spectrum and no point spectrum. We emphasize that
this is true only for the Helmholtz equation with constant
coefficients.

In our case of non-constant and (looking globally) non-periodic
coefficients, there can be nontrivial solutions to the homogeneous
Helmholtz equation (satisfying also a radiation condition). In
general, such solutions can be localized modes or waves that travel
vertically. An example for the first (corresponding to a point
spectrum of the operator) are standing waves in a photonic crystals
with a point defect, compare e.g.\,\cite{PhotonicCrystals-book},
Chapter 5. In the case of a line defect (or in our situation of an
interface between free space and photonic crystal), one expects
nontrivial solutions travelling along the interface, see
e.g.\,\cite{PhysRevB.44.10961, PhysRevB.69.121402}.  Concerning the
mathematical analysis of defects in a photonic crystal and the
possibility that they support modes (and hence act as a waveguide) see
\cite{AmmariSantosa2004, FigotinKlein1998}. In a vertically periodic
setting it was shown in \cite{HoangRadosz-2014} that a line defect
cannot support bounded modes.

The strong uniqueness result of \cite {FlissJoly2016} is fitting in
this background: In a situation where the surrounding medium is
perfectly periodic, the radiation condition of \cite {FlissJoly2016}
implies uniqueness for non-singular frequencies.  Vertical waves in
the crystal are excluded by the non-singularity assumption, waves
along a line defect and localized waves are excluded by the absence of
defects.

Instead, our uniqueness result for the transmission problem must deal
with the fact that the interface can support nontrivial
solutions. Furthermore, we want to admit also singular frequencies
(allowing for vertical waves in the crystal).  Our uniqueness result
states: Imposing the new radiation condition in photonic crystals, for
non-singular frequencies, every homogeneous solution has a vanishing
Bloch measure.  Loosely speaking: the solution vanishes far away from
the interface.  For general frequencies, the radiating solution may
contain vertical waves. See Theorem \ref {thm:uniqueness} for both
results.

\smallskip A more technical problem will accompany us along the way to
a radiation condition and to the uniqueness result: The geometry is
not globally periodic and the solution $u$ is not periodic (and $u$
is, in general, not periodic on any rectangle in the right half
plane).  For this reason, neither a Bloch transformation of $u$ nor a
periodic Bloch expansion of $u$ are meaningful as such. We will have
to truncate $u$ on large squares at the far right and consider the
Bloch expansion of the result. We must use large squares in order to
achieve that the truncation process introduces only small errors.

Bloch measures (as used e.g.\,in \cite{Allaire-Conca-1998},
pp. 182-183) are the appropriate tool for the limit analysis, which is
necessary for the following reason: A periodic Bloch expansion uses a
discrete set of frequences $j$. In general, not even the elementary
frequency condition $\mu_0(j) = \omega^2$ (the Bloch wave frequency
coincides with the frequency of the Helmholtz equation) can be
satisfied in a discrete set of frequencies $j$. For this reason, we
cannot expect that the Bloch expansion of $u$ (at a finite distance)
satisfies conditions such as $\mu_0(j) = \omega^2$. Instead, we must
introduce a limiting object (the Bloch measure). Our aim is to derive
properties of this limiting object.

Regarding other mathematical approaches to related problems, we
mention \cite{MR2533955, MR2847530}, where the authors investigate
diffraction effects in time-dependent equations.  In
\cite{Allaire-Conca-1998}, the spectrum of an elliptic operator in a
periodic medium is investigated. We use some methods of
\cite{Allaire-Conca-1998}, in particular in the pre-Bloch
expansion. Moreover, the above mentioned problem of waves that are
concentrated at the interface of the photonic crystal has a
counterpart in \cite{Allaire-Conca-1998}: The part of the spectrum
that is related to the boundary layer cannot be characterized
explicitly (in the sense of \cite{Allaire-Conca-1998}, where the
sequence $\eps_i\to 0$ is fixed, and in contrast to \cite
{CastroZuazua1996}, where the sequence $\eps_i\to 0$ is chosen
appropriately).

\smallskip We close this section with more references to negative
refraction effects. Negative refraction can be a consequence of a
negative index material, see \cite{Pendry2000} for the effect and
\cite{BFe2, BouchitteSchweizer-Max, ChenLipton2010,
  Lamacz-Schweizer-Max, Lamacz-Schweizer-Neg} for rigorous results,
obtained with the tools of homogenization theory.  In \cite
{EfrosPokrovsky-SolidState-2004, Pokrovsky2003333}, the negative
refraction effect is explained in the spirit of negative index
materials.  But negative refraction can also occur without a negative
index material, see \cite {PhysRevB.65.201104}. We note that the
photonic crystals in \cite {EfrosPokrovsky-SolidState-2004} and in
\cite {PhysRevB.65.201104} are identical and that they do not have a
negative effective index in the sense of homogenization. With the work
at hand we support the line of argument of \cite
{PhysRevB.65.201104}. 

\subsection{Main results}
\label{ssec.main}

Throughout this article we consider the following parameters as fixed:
The frequency $\omega>0$, the height $h>0$ of the waveguide, the
periodicity length $\eps>0$ with $K = h/\eps \in \N$, and the wave
number $k\in \R^2$ of the incident wave with $k_2 h\in \eps\Z$,
$k_1>0$ and $4\pi^2 |k|^2/\eps^2 = \omega^2$. The underlying domain is
$\Omega: = \R\times (0,h)$ and the coefficient field is $a = a^\eps:
\Omega\to \R$. We assume $0<a_*\le a(x) \le a^*<\infty$ $\forall x\in
\Omega$ and $a\equiv 1$ on $\{x_1< 0\}$, but the latter assumption is
not essential.  We demand $a\in C^1$ with $\eps$-periodicity with
respect to $x_1$ and $x_2$ on $\{x_1 > 0\}$.

We use Bloch expansions of the solution. Let us give a description of
our results, where the
superscript ``$\pm$'' indicates that we study $\pm x_1 > 0$. The Bloch
expansion uses two indices, $m\in \N_0 = \{0,1,2,...\}$ numbers the
eigenfunctions in the periodicity cell and the Bloch number $j\in Z:=
[0,1]^2$ measures the phase shift along one periodicity cell. We
collect the two indices in one index as $\lambda := (j,m)\in I:=
Z\times \N_0$. To every $\lambda\in I$ we associate a Poynting number
$P^\pm_\lambda\in \R$, see \eqref {eq:Poynting-def}. For the Bloch wave $U^\pm_\lambda$
with index $\lambda$, the number $P^\pm_\lambda$ is a measure for the
flux of energy in positive $x_1$-direction. 

We introduce the outgoing wave condition (on the right)
\begin{equation}
   \label{eq:outgoingright-intro}
   \meanint_{RY_\eps}\left|\Pi^+_{<0}(u^+_R)\right|^2\to 0\quad
   \text{ as }\quad  R\to \infty\,.
\end{equation}
Here $u^+_R$ is, up to periodic extensions, the function
$u^+_R(x_1,x_2) = u(R\eps + x_1, x_2)$.  The map $\Pi^+_{<0}$ is a
projection onto those Bloch waves that correspond to an energy flux to
the left (i.e. incoming waves, $P^+_\lambda < 0$). The precise
description is given in Definition \ref {def:outgoingwave}.  The
outgoing condition on the left is analogous; Bloch waves that
correspond to an energy flux to the right (i.e. incoming waves,
$P^-_\lambda>0$) are excluded on the far left.

Our results are formulated with the help of index sets.  Waves with
vertical energy flux (or no energy flux) correspond to $\lambda\in I
:=Z\times \N_0$ in
\begin{equation*}
   I^\pm_{=0} := \left\{\lambda\in I\, |\, P^\pm_{\lambda} = 0 \right\}\,,
\end{equation*}
and, for a given $m\in\N_0$, to $j\in Z = [0,1]^2$ in the index set
\begin{equation*}
   J^\pm_{=0,m} := \left\{j\in Z\, |\, P^\pm_{(j,m)} = 0 \right\}
   = \left\{j\in Z\, |\, (j,m)\in I^\pm_{=0} \right\}\,.
\end{equation*}

The statements below are meaningful for general frequencies
$\omega>0$.  Unfortunately, we are only able to prove theorems for
moderate frequencies, as expressed in the following assumption. It
demands that the frequency of the wave is below the energy band
corresponding to the index $m=1$ (below the second band).
\begin{assumption}[Smallness of the frequency]
   \label{ass:smallnessfrequency}
   We assume on the coefficient $a$ and the frequency $\omega$ that
   \begin{equation}
     \label{eq:smallnessfrequency}
     \omega^2 < \inf_{ j\in Z,\, m\geq 1} \, \mu^+_m(j)\,,
   \end{equation}
   and $\omega^2 < \inf_{ j\in Z,\, m\geq 1} \, \mu^-_m(j)$, where
   $\mu^\pm_m(j)$ are the Bloch-eigenvalues.
\end{assumption}

Our results concern solutions $u$ of the transmission problem,
specified as follows.

\begin{problem}[Transmission problem]
  \label{prob:Transmission}
  We say that $u\in H^1_\loc(\Omega)$ solves the scattering problem if
  it satisfies the Helmholtz equation \eqref {eq:Peps} in $\Omega =
  \R\times (0,h)$ with $h = K \eps$ and periodic boundary conditions
  in the $x_2$-variable.  We furthermore assume that it is generated
  by the incoming wave $U_\inc$ of \eqref{eq:u-inc} in the following
  sense: $u$ satisfies the outgoing wave condition \eqref
  {eq:outgoingright} on the right and the difference $u - U_\inc$
  satisfies the outgoing wave condition \eqref {eq:outgoingleft} on
  the left.
\end{problem}

Our uniqueness result characterizes the Bloch measures
$\nu^\pm_{l,\infty}$ of a difference of two solutions (the Bloch
measures are introduced in Definition \ref
{def:limitingBlochmeasure}). The theorem below yields that, for large
values of $|x_1|$, the difference of two solutions does not contain
Bloch waves with an eigenvalue index larger than $0$. Furthermore,
only those waves can appear that satisfy all of the following three
requirements: They correspond to the imposed frequency $\omega$, they
correspond to vertically periodic waves, they transport energy in
vertical direction.

\begin{theorem}[Uniqueness]\label{thm:uniqueness}
  Let Assumption \ref {ass:smallnessfrequency} on the frequency
  $\omega$ be satisfied. For the incoming wave $U_\inc$ with wave
  vector $k$, let $u$ and $\tilde u$ be two solutions of the
  transmission Problem \ref{prob:Transmission}.  For $l\in \N_0$, let
  $\nu^\pm_{l,\infty}$ be the Bloch measures that are generated by the
  difference $v := u-\tilde u$. Then:
   \begin{align}
     \label{eq:Bloch-measure-difference-1}
     \nu^\pm_{l,\infty} &= 0 \quad \text{for }\, l\ge 1\,,\\
     \supp{(\nu^\pm_{0,\infty})} &\subset \left\{ j\in Z\, |\,
       \mu_0^\pm(j) = \omega^2\,,\,
       j_2\in \Z/K  \right\} \cap J^\pm_{=0,0}\,.
   \end{align}
\end{theorem}

An immediate consequence of Theorem \ref {thm:uniqueness} is the
following result for frequencies $\omega$ that do not support vertical
waves.

\begin{corollary}[Uniqueness for non-singular
   frequencies]\label{cor:uniqueness}
   Let the situation be as in Theorem \ref {thm:uniqueness} and let
   the frequency $\omega$ be non-singular in the sense that $\left\{
     j\in Z |\, \mu_0^\pm(j) = \omega^2\,,\, j_2\in \Z/K \right\} \cap
   J^\pm_{=0,0} = \emptyset$. Then the difference $v := u-\tilde u$ of
   two solutions of the transmission Problem \ref{prob:Transmission}
   has a vanishing Bloch measure.
\end{corollary}

Our second main result shows that the transmission of an incoming
wave occurs in such a way that two quantities are conserved: The
vertical wave number and the energy.

\begin{theorem}[Transmission conditions]
   \label{thm:VertWaveNumber}
   Let Assumption \ref {ass:smallnessfrequency} be satisfied, let $k$
   be the wave vector of the incoming wave $U_\inc$. Let $u$ be a
   solution of the transmission problem \ref{prob:Transmission} and
   let $\nu^\pm_{l,\infty}$, with $l\in\N_0$, be the Bloch measures
   that are generated by $u$. Then $\nu^\pm_{l,\infty} = 0$ for $l\ge
   1$ and
   \begin{equation}
     \supp(\nu^\pm_{0,\infty})\, \subset\,
     \left\{ j\in Z\, |\,
       \mu_0^\pm(j) = \omega^2
       \,,\, j_2\in \Z/K \right\}
     \medcap \left(
       \left\{ j\in Z\, |\,
         j_2 = k_2\right\}\phantom{\int} \!\!\!\!\!
       \cup\  J^\pm_{=0,0} \right)\,.
   \end{equation}
\end{theorem}

As above, we have the following corollary for non-singular
frequencies.

\begin{corollary}[Transmission condition for non-singular
   frequencies]\label{cor:uniqueness}
   Let the situation be as in Theorem \ref {thm:VertWaveNumber} and let
   the frequency $\omega$ be non-singular. Then the Bloch measure
   $\nu^\pm_{0,\infty}$ of $u$ satisfies
   \begin{equation}
     \supp(\nu^\pm_{0,\infty})\ \subset\
     \left\{ j\in Z\, |\,
       \mu_0^\pm(j) = \omega^2\, \text{and }
       j_2 = k_2\right\} \,.
   \end{equation}
\end{corollary}

\subsection{Further comments on the main results}

\paragraph{\em On the uniqueness result.}  We recall that we expect
the existence of solutions that are supported on the interface between
photonic crystal and free space. For this reason, uniqueness results
can only provide information ``far away from the interface'',
i.e.\,information on the Bloch measure.

A weakness of our uniqueness results concerns Assumption \ref
{ass:smallnessfrequency}: Our results are proven under the assumption
that the underlying frequency $\omega$ is in the first band (more
precisely: below the second band).  Our conjecture is that our
uniqueness result remains valid for arbitrary frequencies, stating
that $\supp(\nu^\pm_{l,\infty}) \subset \left\{ j |\, \mu_l^\pm(j) =
  \omega^2, j_2\in \Z/K \right\} \cap J^\pm_{=0,l}$ for every $l\ge
0$.  Due to a lack of orthogonality properties in the sesquilinear
form $b$ (see Section \ref{sec.uniqueness}), we must exploit the
frequency assumption in our uniqueness proof.

\paragraph{\em Relations to Fliss and Joly \cite {FlissJoly2016}.} The
contribution \cite {FlissJoly2016} contains strong results: 1. A
uniqueness result in the classical form (due to the absence of an
interface that can support waves and due to the restriction to
non-singular frequencies).  2. An existence result, based on a
limiting absorption principle. We note that also the existence result
of \cite {FlissJoly2016} uses global Floquet-Bloch transformations and
is therefore not easily adaptable to our setting. We remark that our
outgoing wave condition is weaker than the one of \cite
{FlissJoly2016}, see Lemma \ref {lem:OutgoingwaveSingleWave}. This
means that, apart from the problems due to the non-periodic geometry,
an existence proof should be simpler for our outgoing wave condition.

We mention at this place that our outgoing wave condition differs in
one point with all existing radiation conditions: Our condition does
not use explicitely the frequency $\omega$. We regard this as an
advantage: our condition might be applicable also in time-dependent
problems.

\paragraph{\em A possible scaling in $\eps>0$.}  In all our theorems
we keep the length scale $\eps>0$ fixed. In other words: the
wave-length $1/\omega$ and the periodicity length $\eps$ are both of
order $1$.  It is very interesting to analyze the behavior of light in
small micro-structures, i.e.\,to analyze the limit $\eps\to 0$. The
limit can be performed in two settings: In the classical
homogenization problem, one keeps $\omega$ (and hence the wave-length)
fixed and analyzes the behavior of solutions $u = u^\eps$ as $\eps\to
0$. This approach was carried out e.g.\,in \cite{BFe2,
  BouchitteSchweizer-Max, ChenLipton2010, Lamacz-Schweizer-Max,
  Lamacz-Schweizer-Neg}.

The second setting regards the limit $\eps\to 0$ in a situation where
the wave-length of the incoming wave is also of order $\eps$. This is
the scaling that is suggested by our notation in \eqref {eq:u-inc},
which corresponds to a frequency $\omega = \omega^\eps = \eps^{-1}
\omega^*$. Loosely speaking, our Theorem \ref {thm:VertWaveNumber}
yields in this scaling: Solutions $u^\eps$ to the scattering problem
with incoming wave \eqref {eq:u-inc} for fixed $k$ consist, at a fixed
distance $x_1 > 0$ from the interface and in the limit $\eps\to 0$,
only of Bloch waves that correspond to the frequency $\omega^\eps$ and
to the wave number $k_2$ (up to vertical waves).

\paragraph{\bf Outline of this contribution.}
Bloch expansions are described in Section \ref {sec.Bloch}. In Section
\ref {sec.outgoingwave} we define energy flux numbers and
corresponding index sets; these are used to define the new outgoing
wave condition. In Section \ref {sec.uniqueness} we define Bloch
measures, Theorems \ref {thm:uniqueness} and \ref {thm:VertWaveNumber}
are shown in Section \ref{ssec.Bloch-uniqueness}.

\section{Bloch expansions}
\label{sec.Bloch}

\subsection{Pre-Bloch expansions}
We start our analysis with a discrete expansion. This discrete
expansion is the first stage of a Bloch expansion and closely related
to the Floquet-Bloch transform. We apply it to the $h$-periodic
function $u(x_1,\cdot)$. The subsequent result appears as Lemma 4.9 in
\cite{Allaire-Conca-1998}.

\begin{lemma}[Vertical pre-Bloch expansion]
  Let $K\in\N$ be the number of periodicity cells and let $h=\eps K$
  be the height of the strip $\R\times (0,h)$. Let $u\in
  L^2_\loc(\R\times (0,h);\C)$ be a function. Then $u$ can be expanded
  uniquely in periodic functions with phase-shifts: With the finite
  index set $Q_{K} := \{0,\frac{1}{K}, \frac{2}{K}, \dots,
  \frac{K-1}{K}\}$ we find
  \begin{equation}
    \label{eq:discrete-u-1}
    u(x_1, x_2) 
    = \sum_{j_2\in Q_{K}} \Phi_{j_2}(x_1, x_2)\, e^{2\pi i j_2 x_2/\eps}\,,
  \end{equation}
  where each function $\Phi_{j_2}(x_1, \cdot)$ is $\eps$-periodic.
  The equality \eqref {eq:discrete-u-1} holds in
  $L^2_\loc(\R\times(0,h);\C)$.
\end{lemma}

\begin{proof}[Sketch of proof]
  We sketch a proof (different from the one chosen in
  \cite{Allaire-Conca-1998}), considering only $u = u(x_2)$ and
  $h=1$. Expanding $u$ in a Fourier series, we may write
  \begin{equation}
    \label{eq:u-Fourier}
    u(x_2)  = \sum_{k_2\in \eps\Z} \beta_{k_2}\, e^{2\pi i k_2 x_2/\eps}\,.
  \end{equation}
  For every $j_2\in \eps\N_0$ with $j_2<1$ (i.e.\,for every $j_2\in
  Q_K$) we set
  \begin{equation}
    \label{eq:Phi-pre-Bloch}
    \Phi_{j_2}(x_2) := \sum_{k_2\in j_2 + \Z} \beta_{k_2}\, 
    e^{2\pi i (k_2-j_2) x_2/\eps}\,.
  \end{equation}
  With this choice, each $\Phi_{j_2}$ is $\eps$-periodic and \eqref
  {eq:discrete-u-1} is satisfied.
\end{proof}

For the above pre-Bloch expansion we define the projection on a
vertical wave number $k_2$ as follows.

\begin{definition}[Vertical pre-Bloch projection $\Pi^\ver_{k_2}$]
  \label{def:verticalprojectionpreBloch}
  Let $u\in L^2_\loc(\R\times (0,h);\C)$ with $h=\eps K$ be a function
  on a strip and let $k_2\in Q_{K}$ be a vertical wave number. Then,
  expanding $u$ as in \eqref {eq:discrete-u-1}, we set
  \begin{align}
    \label{eq:vertProjection-1}
    \Pi^\ver_{k_2}u(x_1,x_2) := 
    \Phi_{k_2}(x_1, x_2)\, e^{2\pi i k_2 x_2/\eps}\,.
  \end{align}
\end{definition}

The projection is an orthogonal projection: For $\eps$-periodic
functions $\Phi$ and $\tilde\Phi$ and indices $k_2\neq\tilde k_2$
there holds $\int_0^h \overline{\Phi(x_2)} e^{-2\pi i \tilde k_2
  x_2/\eps} \tilde\Phi(x_2) e^{2\pi i k_2 x_2/\eps}\, dx_2 = 0$ by
Lemma \ref {lem:orthweight}.

We will later use the following fact: If $u$ is a solution of the
scattering problem with incident vertical wave number $k_2$, then also
the projection $\Pi^\ver_{k_2}u$ is a solution of the scattering
problem. Together with a uniqueness result for solutions, we can
conclude from this fact that the vertical wave number is conserved in
the photonic crystal.  

Below, we have to deal with the following situation: For a function
$u$ on a strip with height $h$, we can perform a pre-Bloch
expansion. We may also extend $u$ periodically in the vertical
direction and perform a pre-Bloch expansion of the extended function
on a wider strip. We find that both constructions yield the same
result.

\begin{remark}[Vertical pre-Bloch expansion of a periodically extended
  function] \label{rem:Blochdifferentperiod}
  Let $K=h/\eps\in\N$ denote the number of periodicity cells in vertical
  direction and let $u\in L^2_\loc(\R\times (0,h))$ be a function with
  vertical pre-Bloch expansion
  \begin{align*}
    u(x_1,x_2)=\sum_ {j_2\in Q_{K}} \Phi_{j_2}(x_1,x_2)\,
    e^{2\pi i j_2 x_2/\eps}.
  \end{align*}
  Let $R\in \N$ be a multiple of $K$ and let $\tilde u$ be the
  periodic extension of $u$ to the interval $(0,\eps R)$ in
  $x_2$-direction.  Then $\tilde u\in L^2_\loc(\R\times (0,\eps R))$
  has the vertical pre-Bloch expansion
  \begin{align}
    \label{eq:PreBlochbiggerperiod}
    \tilde u(x_1,x_2)=\sum_{\tilde j_2\in Q_R} \tilde\Phi_{\tilde
      j_2}(x_1,x_2)\,e^{2\pi i \tilde j_2 x_2/\eps},
  \end{align}
  where the coefficients according to the finer grid $Q_R$ satisfy
  \begin{align*}
    \tilde\Phi_{\tilde j_2}(x)=
    \begin{cases}
       \begin{array}{rcl}
         0 & \text{ if } \tilde j_2\not\in Q_{K}\,,\\
         \Phi_{\tilde j_2}(x) & \text{ if } \tilde j_2\in Q_{K}\,.
       \end{array}
     \end{cases}
   \end{align*}
\end{remark}

The statement follows immediately from the uniqueness of the pre-Bloch
expansion. Remark \ref {rem:Blochdifferentperiod} explains our choice
concerning scalings: Given a sequence of functions $u_R$, defined on
a sequence of increasing domains, at first sight, one might find it
natural to rescale $u_R$ to a standard domain and to analyze the
sequence of rescaled functions. Instead, we work with the sequence
$u_R$ on increasing domains. In this way, one index $j\in Z$ always
refers to the same elementary wave, which allows to investigate the
Bloch measure limit.

\paragraph{Pre-Bloch expansion in two variables.} For a function $u$
that is defined on a rectangle and that is periodic in both
directions, the pre-Bloch expansion in two variables can be defined by
expanding first in one variable and then in the other.

For functions $u$ on $\R\times (0,h)$ the situation is more difficult,
since $u$ is not periodic in $x_1$-direction. In order to expand in
both directions, we truncate $u$ with a cut-off function $\eta :
\R\times [0,h]$ with compact support. For convenience, we assume that
the support of $\eta$ is contained in the square $[0,h]\times [0,h]$.

The truncation of $u$ is defined as $w(x):= u(x)\, \eta(x)$.  We
expand $w$ (on the square $[0,h]\times [0,h]$) in both directions in a
pre-Bloch expansion, using the vector $j = (j_1, j_2)\in Q_{K}\times
Q_{K}$ and $x = (x_1,x_2)$:
\begin{equation}
  \label{eq:discrete-v-eps-2}
  w(x) 
  = \sum_{j\in Q_{K}\times Q_{K}} \Phi_{j}(x)\, e^{2\pi i j\cdot x / \eps}\,.
\end{equation}
The functions $\Phi_j = \Phi_{(j_1, j_2)}$ are now $\eps$-periodic in
both variables. Due to orthogonality there holds ($h=\eps K$)
\begin{align*}
  \frac{1}{(\eps K)^2} 
  \|w\|^2_{L^2(K Y_\eps)}=\sum_{j\in Q_{K}\times Q_{K}} 
  \meanint_{Y_\eps} \left|\Phi_{j}\right|^2, 
\end{align*}
where $\textmean_{Y_\eps}\left|\Phi_{j}\right|^2
:=\frac{1}{|Y_\eps|}\int_{Y_\eps}\left|\Phi_{j}\right|^2$ denotes the
mean value.

\subsection{Bloch expansion}

With the help of the pre-Bloch expansion we construct now the Bloch
expansion. This step consists in developing each of the periodic
functions $\Phi_{j}$ for $j = (j_1, j_2)$ in terms of eigenfunctions
of the operator
\begin{equation}\label{eq:L_j-operator}
  \calL_j^+
  := -\left(\nabla + 2\pi ij/\eps\right) 
  \cdot \left(a^\eps(x) \left(\nabla+ 2\pi ij/\eps\right)\right)\,.
\end{equation}
The operator $\calL_j^+$ acts on complex-valued functions on the cell
$Y_\eps$ with periodic boundary conditions. It appears in the analysis
of \eqref {eq:Peps} for the following reason: Let $\Psi_{j}^+$ be an
eigenfunction of $\calL_j^+$ with eigenvalue $\mu^+(j)$; then there
holds
\begin{equation*}
  -\nabla \cdot \left(a^\eps(x) \nabla [\Psi_{j}^+ e^{2\pi i j\cdot x / \eps}]\right)
  = [\calL_j^+\Psi_{j}^+] e^{2\pi i j\cdot x / \eps} 
  = \mu^+(j)\, [\Psi_{j} e^{2\pi i j\cdot x / \eps}]\,.
\end{equation*}
We see that $\Psi_{j}^+ e^{2\pi i j\cdot x / \eps}$ is a solution of
the Helmholtz equation on the right half-plane if and only if
$\mu^+(j)=\omega^2$.

We have to distinguish between $x_1>0$ and $x_1<0$. On the right, the
expansion is performed with $\calL_j^+$ as above, with the periodic
coefficient $a^\eps = a^\eps(x)$. On the left, expansions are
performed according to $a^\eps\equiv 1$ with the operator $\calL_j^-
:= -\left(\nabla + 2\pi ij/\eps\right) \cdot \left(\nabla+ 2\pi
  ij/\eps\right)$. The result is a classical Fourier expansion of the
solution.

\begin{definition}[Bloch eigenfunctions] Let $j\in[0,1]^2$ be a fixed
  wave vector.  We denote by $\left(\Psi^+_{j,m}\right)_{m\in\N_0}$
  the family of eigenfunctions of the operator $\calL_j^{+}$ of \eqref
  {eq:L_j-operator}. The labelling is such that the corresponding
  eigenvalues $\mu^{+}_m(j)$ are ordered, $\mu_{m+1}(j)\geq \mu_m(j)$
  for all $m\in\N_0$.  Similarly, $\left(\Psi^-_{j,m}
  \right)_{m\in\N_0}$ is the family of eigenfunctions of the operator
  $\calL_j^{-}$ and $\mu^{-}_m(j)$ are the corresponding eigenvalues.
  We normalize with $\textmean_{Y_\eps}|\Psi^\pm_{j,m}|^2 = 1$.
\end{definition}

A standard symmetry argument yields that, after an appropriate
orthonormalization procedure for multiple eigenvalues, all functions
$\Psi^\pm_{j,m}(x)\, e^{2\pi i j\cdot x / \eps}$ with $j\in Q_K$ are
orthonormal in the space $L^2_\sharp(KY_\eps; \mathbb{C}) =
L^2(KY_\eps; \mathbb{C})$ (the sharp symbol is sometimes used to
indicate that one thinks of periodic functions, but, of course, in the
case of $L^2$, the periodicity does not alter the function space). On
the left hand side (i.e.\,for $x_1<0$, denoted with the superscript
``-''), the Bloch eigenfunctions are harmonic waves and the Bloch
expansion coincides with a Fourier expansion. We collect properties on
the left half-domain in Remark \ref {rem:Left-Fourier}.

\begin{lemma}[Bloch expansion]\label{lem:Bloch-expansion}
  Let $K\in\N$ be the number of cells in each direction, let $u\in
  L^2(K Y_\eps;\C)$ be a function on the square $(0,K\eps)\times
  (0,K\eps)$. Expanding $u$ in a pre-Bloch expansion and then
  expanding each $\Phi_j$ in eigenfunctions $\Psi^+_{j,m}$ we obtain,
  with coefficients $\alpha_{j,m}^+\in\C$,
  \begin{align*}
    u(x)&= \sum_{j\in Q_{K}\times Q_{K}} \sum_{m=0}^\infty
    \alpha^+_{j, m}  \Psi^+_{j, m}(x)\, e^{2\pi i j\cdot x / \eps}\,,
  \end{align*}
  and similarly, for an expansion corresponding to constant
  coefficients $a^\eps\equiv 1$,
  \begin{align*}
    u(x)&= \sum_{j\in Q_{K}\times Q_{K}} \sum_{m=0}^\infty
    \alpha^-_{j, m}  \Psi^-_{j, m}(x)\, e^{2\pi i j\cdot x / \eps}\,.
  \end{align*}
\end{lemma}
 
To shorten notation, we will use the multi-index $\lambda = (j,m)$ in
the index-set $I_{K}:= \{(j,m) | j\in Q_{K}\times Q_{K},\ m\in\N_0\}
\subset I:=Z\times\N_0$. Abbreviating additionally
\begin{align}
  \label{eq:abbreviationUlambda}
  U^\pm_\lambda(x) := \Psi^\pm_\lambda (x)\, e^{2\pi i
    j\cdot x/\eps}\,,
\end{align}
we may write the formulas of Lemma \ref {lem:Bloch-expansion} as
\begin{equation}
  \label{eq:discrete-v-eps-Bloch}
  u(x) = \sum_{\lambda=(j,m)\in I_{K}}
  \alpha^\pm_\lambda \Psi^\pm_\lambda (x)\, e^{2\pi ij\cdot x/\eps}
  = \sum_{\lambda \in I_{K}} \alpha^\pm_\lambda U^\pm_\lambda (x)\,.
\end{equation}
The expansion holds for the basis functions $U^+_\lambda$ with
coefficients $\alpha^+_\lambda$ and for the basis functions
$U^-_\lambda$ with coefficients $\alpha^-_\lambda$. Moreover, due to
$L^2$-orthonormality of the functions $U_\lambda^\pm$, with $h = \eps
K$ and $K Y_\eps = (0,h) \times (0,h)$,
\begin{align}
  \label{eq:L2normexpansion}
  \frac{1}{(\eps K)^2}\|u\|^2_{L^2(K Y_\eps)} = \sum_{\lambda \in
    I_{K}}|\alpha^\pm_\lambda|^2\,.
\end{align}

\section{Outgoing wave condition}\label{sec.outgoingwave}

\subsection{Poynting numbers and projections}

Let $\lambda=(j,m)\in I$ be an index and let $U^\pm_\lambda$ be the
corresponding Bloch function.  Denoting by $e_1= (1,0) \in \R^2$ the
first unit vector, we connect to $\lambda\in I$ the real numbers
\begin{align}
  \begin{split}
    \label{eq:Poynting-def}
    P^+_\lambda &:= \Im\, \meanint_{Y_\eps} \bar U^+_\lambda (x)\,
    e_1\cdot \left[a^\eps(x)\nabla U^+_\lambda(x)\right]\, dx\,,\\
    P^-_\lambda &:= \Im\, \meanint_{Y_\eps} \bar U^-_\lambda (x)\, e_1\cdot \nabla
    U^-_\lambda(x)\, dx\,.
  \end{split}
\end{align}
The number $P^+_\lambda$ is related to the Poynting vector of the
Bloch eigenfunction $U^+_\lambda$. It corresponds (up to a positive
constant) to the horizontal group velocity of this eigenfunction and
measures its energy flux in horizontal direction: In the case
$P^+_\lambda > 0$, the energy of the wave is travelling to the right,
in the case $P^+_\lambda < 0$, the energy of the wave is travelling
to the left. 

Let us point out the relation to Maxwell's equations: If $u$ denotes the
out-of-plane magnetic field, i.e.\,$H = (0,0,u)$, then the electric
field is $(E_1, E_2, 0)$ with $E_1 = (-i\omega\epsilon)^{-1} \del_2 u$
and $E_2 = (i\omega\epsilon)^{-1} \del_1 u$ where $\epsilon$ is the
permittivity of the medium.  The complex Poynting vector is $P =
\frac12 E\times \bar H$, so the real part of its horizontal component
is $\Re (e_1 \cdot P) = \Re(\frac12 \bar H_3 E_2) = (2\omega)^{-1}
\Re(-i\, \bar u\, \epsilon^{-1} \del_1 u) = (2\omega)^{-1} \Im(\bar
u\, a\del_1 u)$, where we used that the coefficient $a =
\epsilon^{-1}$ is the inverse permittivity. Our expression in \eqref
{eq:Poynting-def} coincides up to the factor $2\omega$ with an
integral of this expression.  Since $P$ represents a local energy
flux, the physical quantity of a total energy flux is a surface
integral over $P$. In fact, for solutions $U_\lambda$ of a Helmholtz
equation, the surface integral is independent of the position of the
surface. Hence our volume integral in \eqref {eq:Poynting-def} indeed 
coincides with the physical quantity of a surface integral.

{\em The index set for $\lambda$:} In our construction, we fix the
height $h>0$ of the domain and the periodicity length $\eps = h/K$,
the Bloch expansion is performed in this fixed geometry. As a
consequence, we consider only indices $\lambda = (j,m) \in I_K$, the
frequency parameter $j$ must lie in the discrete set $Q_K\times Q_K
\subset Z$.  On the other hand, for arbitrary $j\in Z$, we can still
consider the functions $\Psi^\pm_{j,m}$ and $U^\pm_\lambda$. They do
not depend on $K$, hence also the values $P^\pm_\lambda$ are
independent of $K$.

\begin{definition}[Index sets and projections]
  \label{def:Projections}
  We define the set of indices corresponding to right-going waves in
  $x_1>0$ as
  \begin{equation}
    \label{eq:index-sets}
    I^+_{>0} := \left\{ \lambda\in I\ |\  P^+_\lambda > 0\right\}\,.
  \end{equation}
  The index sets $I^-_{>0}$, $I^\pm_{<0}$, $I^\pm_{\ge 0}$,
  $I^\pm_{\le 0}$, $I^\pm_{=0}$ are defined accordingly.

  For $K\in\N$ we define the projections $\Pi^\pm_{>0}$ as follows:
  Let $u\in L^2(K Y_\eps;\C)$ be a function with the discrete Bloch
  expansion
  \begin{equation*}
    u(x) = \sum_{\lambda\in I_{K}} \alpha^\pm_\lambda
    U^\pm_\lambda(x)\,.
  \end{equation*}
  Then we set
  \begin{equation*}
    \Pi^\pm_{>0}u(x) :=\sum_{\lambda\in I_{K}\cap\,
      I^\pm_{>0}}\alpha^\pm_\lambda U^\pm_\lambda(x)\,.
  \end{equation*}
  With this definition, $\Pi^\pm_{>0}$ are the projections onto
  right-going Bloch-waves. The projections $\Pi^\pm_{<0},\Pi^\pm_{\geq
    0}$, $\Pi^\pm_{\leq 0}$, and $\Pi^\pm_{=0}$ are defined
  accordingly.

  For $k_2\in Q_{K}$ and $l\in\N_0$, the ``vertical'' projection
  $\Pi^{\ver,\pm}_{k_2}$ and the ``eigenvalue'' projection
  $\Pi^{\ev,\pm}_{l}$ are defined by
  \begin{align*}
    \Pi^{\ver,\pm}_{k_2}u(x)&:=\sum_{\lambda\in \{(j,m)\in I_{K}|\,j_2=k_2\}} 
    \alpha^\pm_\lambda U^\pm_\lambda(x)\,,\\
     \Pi^{\ev,\pm}_{l}u(x)&:=\sum_{\lambda\in \{(j,m)\in I_{K}|\,m=l\}}
     \alpha^\pm_\lambda U^\pm_\lambda(x)\,.
  \end{align*}
\end{definition}

Note that the projections $\Pi^{\ver,\pm}_{k_2}$ of the discrete Bloch
expansion indeed coincide with the projection $\Pi^\ver_{k_2}$ of the
corresponding vertical pre-Bloch expansion of Definition
\ref{def:verticalprojectionpreBloch}. The vertical projection is
independent of $K$ in the sense that a periodically extended $u$ with
a larger value of $K$ has the same projection, compare Remark \ref
{rem:Blochdifferentperiod}.

\subsection{Bloch expansion at infinity and outgoing wave condition}

We can now formulate the outgoing wave condition for a solution $u$ of
the Helmholtz equation \eqref{eq:Peps}. The loose description of our
outgoing wave condition (on the right) is: The Bloch expansion of $u$
does not contain Bloch waves that transport energy to the left.

For a rigorous definition we must deal with the problem that $u$ is
not necessarily periodic in $x_1$-direction. Our solution to this
problem is to consider $u$ on large domains (which reduces the effects
of non-periodicity) and to employ a truncation procedure.
Furthermore, we want to formulate a condition that characterizes $u$
for large values of $x_1$. For these two reasons, we consider $u(x_1,
x_2)$ for $x_1 \in (R\eps, 2R\eps)$ with a large natural number $R >>
K$.

In order to construct a function on a large domain, we consider the
periodic extension of $u$ in vertical direction and restrict
afterwards the function to a large rectangle. For convenience of
notation, we restrict our analysis to squares.

\begin{definition}[Bloch expansion far away from the interface]
  \label{def:restrictionbox}
  Let $u\in L^2_\loc(\R\times (0,h);\C)$ be a function on the infinite
  strip with height $h = \eps K$. Let $R\in \N K$ be a multiple of
  $K$. We define $\tilde u:\R^2\rightarrow \C$ as the $h$-periodic
  extension of $u$ in $x_2$-direction.  We furthermore define
  functions $u_R^\pm : RY_\eps\to \C$ by
  \begin{align}
    u_R^+(x_1,x_2) &:= \tilde u(R\eps+x_1,x_2)\,,\\
    u_R^-(x_1,x_2) &:= \tilde u(-2R\eps+x_1,x_2)\,.
  \end{align} 
  We use the discrete Bloch expansions of the functions $u^\pm_R\in
  L^2_\sharp(RY_\eps;\C)$,
  \begin{equation}
    \label{eq:uRdiscreteBloch}
    u^\pm_R(x) = \sum_{\lambda\in I_{R}} 
    \alpha^\pm_{\lambda,R} U^\pm_\lambda(x)\,.
  \end{equation}
  The coefficients $(\alpha^\pm_{\lambda,R})_{\lambda\in I}$ encode
  the behavior of $u$ for large values of $|x_1|$.
\end{definition}

We are now in the position to define the outgoing wave condition for a
solution $u$ to the Helmholtz equation, using the short notation
$\meanint_{RY_\eps} f := \frac1{|RY_\eps|} \int _{RY_\eps} f$ for
averages of functions.

\begin{definition}[Outgoing wave condition]
  \label{def:outgoingwave}
  For $K\in\N$, $h = K\eps$, and $R\in \N K$, we consider $u\in
  L^2_\loc(\R\times (0,h);\C)$. We say that $u$ satisfies the outgoing
  wave condition on the right if the following two conditions are
  satisfied: $\int_0^h \int_L^{L+1} |u|^2$ is bounded, independently
  of $L\ge0$, and
  \begin{equation}
    \label{eq:outgoingright}
    \meanint_{RY_\eps}\left|\Pi^+_{<0}(u^+_R)\right|^2\to 0
    \text{ as }  R\to \infty\,.
  \end{equation}

  Accordingly, we say that $u$ satisfies the outgoing wave condition
  on the left, if $\int_0^h \int_{L-1}^{L} |u|^2$ is bounded,
  independently of $L\le 0$, and if
  \begin{equation}
    \label{eq:outgoingleft}
    \meanint_{RY_\eps}\left|(\Pi^-_{>0}(u^-_R)\right|^2\to 0
    \text{ as }  R\to \infty\,.
  \end{equation}
\end{definition}

Let us repeat the idea of condition \eqref {eq:outgoingright}: The
function $u$ is considered at the far right by construcing $u^+_R$ as
in Definition \ref {def:restrictionbox}. This function is projected
onto the space of left-going waves. We demand that the $L^2$-averages
of the resulting functions $\Pi^+_{<0}(u^+_R)$ vanish in the limit
$R\to \infty$.

With the expansion \eqref {eq:uRdiscreteBloch} we can write condition
\eqref {eq:outgoingright} equivalently as:
\begin{equation}
\label{eq:utgoingrightequivalent}
  \sum_{\lambda\in I_{R}\cap\,
    I^+_{<0}} |\alpha^+_{\lambda,R}|^2 \to 0\text{ as }  R\to \infty\,.
\end{equation}
Our aim is to show that this definition of an outgoing wave condition
implies uniqueness properties for the scattering problem.

We note that the uniform $L^2$-bounds for large values of $|L|$ imply,
for solutions $u$ of the Helmholtz equation, also uniform bounds for
gradients, see Lemma \ref {lem:Caccioppoli} in the appendix.

\subsection{Truncations and $m\ge 1$-projections}
\label{ssec.truncations}

In the outgoing wave condition, we study the limit $|x_1| \sim R \to
\infty$ and the functions $u^\pm_R$ on large squares $W_R := RY_\eps =
(0,R\eps)^2$ with $|W_R| = (\eps R)^2$. As a measure for typical
values of a function $v$ we use $L^2$-averages on $W_R$ and the
corresponding scalar product,
\begin{align}\label{eq:weighted-norms}
  \la v, w \ra_R :=\meanint _{W_R}  v\cdot \bar w\,.
\end{align}

In the following we denote by $\calL_0 = \calL_0^+ = -\nabla\cdot
\left(a^\eps \nabla\right)$ the elliptic operator of \eqref
{eq:L_j-operator}.  As above, we denote cubes by $W_R = RY_\eps$ and,
by slight abuse of notation, we write $W_{R-1}:=\eps(1,R-1)^2$ for a
smaller cube that has the point $\eps (1,1)$ as its bottom left corner
and $\eps (R-1,R-1)$ as its top right corner. We use a family of
smooth cut-off functions $\eta := \eta_R$ with the properties
\begin{align}
  \label{eq:cutofffct}
  \eta_R\in C^\infty(W_R;\R),\quad \eta_R=1\,\text{on }W_{R-1},\quad
  \|\nabla\eta_R\|_\infty\leq C_0, \quad\|\nabla^2\eta_R\|_\infty\leq
  C_0
\end{align}
for some $R$-independent constant $C_0$ ($\eps>0$ is fixed), and with
compact support in $(0,R\eps)\times (0,R\eps)_\sharp$, where
$(0,R\eps)_\sharp$ indicates the interval with identified end
points. The latter requirement admits sequences $\eta$ with compact
support in $(0,R\eps)\times (0,R\eps)$, but also sequences of
vertically periodic functions $\eta$, in particular functions
$\eta=\eta(x_1)$. In the subsequent proofs we do not indicate the
$R$-dependence of $\eta_R$ and write only $\eta$. We furthermore omit
the superscipts $\pm$, the eigenvalue corresponding to $\lambda =
(j,m)$ is denoted by $\mu_\lambda = \mu_m(j)$.  Constants are allowed
to depend on $\eps>0$.

\begin{lemma}[The effect of truncations] \label{lem:almost-psi} For
  $R\in\N$ let $\eta = \eta_R$ be a family of cut-off functions
  satisfying \eqref{eq:cutofffct}. Let $v_R$ and $w_R$ be sequences of
  functions in $L^2(W_R;\C)$ with $v_R\in H^2(W_R;\C)$. We assume that
  certain averages over boundary strips are bounded:
  \begin{equation}
    \label{eq:UniformBounds}
    \frac1{R} \int_{W_R\setminus W_{R-1}} |v_R|^2 + |\nabla v_R|^2
    \le C_0\,,\qquad \frac1{R} \int_{W_R\setminus W_{R-1}} |w_R|^2
    \le C_0\,,
  \end{equation} 
  with $C_0$ independent of $R$.  Then, with a constant $C$ that is
  independent of $R$:
  \begin{enumerate}
  \item Application of $\calL_0$ to a truncated function:
    \begin{equation}
      \meanint_{W_R}\left| \calL_0 (v_R)\eta - \calL_0(v_R\eta) \right|^2
      =\meanint_{W_R}\left| \calL_0 (v_R)\eta - \sum_{\lambda\in I_R}
        \mu_\lambda\langle v_R\eta, U_\lambda\rangle_R\,
        U_\lambda \right|^2 \leq \frac{C}{R}\,.
      \label{eq:solutions-and-mu-factors-weight}
    \end{equation}
  \item If $\Pi$ is one of the projections of Definition
    \ref{def:Projections}, then
    \begin{equation}
      \label{eq:projectiontruncated}
      \meanint_{W_R} \left|\Pi (w_R)-\Pi (w_R\,\eta)\right|^2 \leq
      \meanint_{W_R} \left|w_R-w_R\,\eta\right|^2 \leq
      \frac{C}{R}\,.
    \end{equation}
  \end{enumerate}
\end{lemma}

\begin{proof}
  In the following, the letter $C$ denotes different constants,
  possibly varying from one line to the next, but always independent
  of $R$.  To prove \eqref{eq:solutions-and-mu-factors-weight}, we
  expand the $L^2$-function $\calL_0 (v_R)\eta$ in Bloch-waves. The
  following calculation uses several times integration by parts; due
  to the $\eta$-factor, no boundary integrals occur. In the first
  equation we use that the coefficient $\alpha_\lambda$ in the
  expansion of $\calL_0(v_R)\eta$ is obtained by taking the scalar
  product with $U_\lambda$ (orthonormality of the $U_\lambda$).
  \begin{align*}
    \calL_0(v_R)\eta
    &=\sum_{\lambda\in I_R}\left\langle \calL_0(v_R)\eta,
      U_\lambda\right\rangle_R U_\lambda
    =\sum_{\lambda\in I_R}\left\langle v_R, \calL_0(\eta U_\lambda) 
    \right\rangle_R U_\lambda\\
    &=\sum_{\lambda\in I_R}\left(\left\langle
        v_R\eta,\calL_0 U_\lambda\right\rangle_R+ \left\langle v_R
        \calL_0(\eta),U_\lambda\right\rangle_R-2\left\langle v_R\,
        a^\eps\nabla\eta,
	\nabla U_\lambda\right\rangle_R\right)U_\lambda\\
    &=\sum_{\lambda\in I_R}\left(\mu_\lambda\left\langle
        v_R\eta,U_\lambda\right\rangle_R 
      - \left\langle v_R  \calL_0(\eta),U_\lambda\right\rangle_R
      +2\left\langle \nabla v_R\cdot  a^\eps\nabla\eta ,
        U_\lambda\right\rangle_R\right)U_\lambda\\
    &= \left(\sum_{\lambda\in I_R}\mu_\lambda\left\langle
      v_R\eta,U_\lambda\right\rangle_R U_\lambda \right) 
  -v_R \calL_0(\eta)+2a^\eps\nabla v_R\cdot\nabla\eta\,,
  \end{align*}
  where in the third equality we exploited $\calL_0
  U_\lambda=\mu_\lambda U_\lambda$ and $\mu_\lambda\in \R$.  The
  contribution of the last two terms can be estimated by
  \begin{align*}
    &\meanint_{W_R}\left|v_R\,\calL_0(\eta)\right|^2
    +\left|2a^\eps\nabla v_R\cdot\nabla\eta\right|^2 \\
    &\quad \leq
    \|\calL_0(\eta)\|^2_{L^\infty(W_R)}\meanint_{W_R}|v_R|^2
    1_{\{\text{supp}(\nabla\eta)\}}
    +\|2a^\eps\nabla\eta\|^2_{L^\infty(W_R)}\meanint_{W_R}|\nabla v_R|^2 
    1_{\{\text{supp}(\nabla\eta)\}}\displaybreak[2]\\
    &\quad \leq \frac{C}{R^2}
    \left(
      \int_{W_R\setminus  W_{R-1}}|v_R|^2
      + 
      \int_{W_R\setminus W_{R-1}}|\nabla v_R|^2 \right)
     \leq \frac{C}{R}\,.
  \end{align*}
  In the second inequality we exploited $\text{supp}(\nabla\eta)
  \subset (W_R\setminus W_{R-1})$, in the last inequality we used the
  uniform bounds \eqref{eq:UniformBounds}. This proves the inquality
  in \eqref {eq:solutions-and-mu-factors-weight}.

  Regarding the equality in \eqref
  {eq:solutions-and-mu-factors-weight} we have to verify that the
  formal equality $\calL_0 w = \sum_\lambda \mu_\lambda\langle w,
  U_\lambda\rangle_R\, U_\lambda$ holds for functions $w\in H^2(W_R)$
  with vanishing boundary data. We find this from
  \begin{equation}\label{eq:L0-applied-to-w}
    \la \calL_0 w, U_{\lambda} \ra \stackrel{(!)}{=} 
    \la  w, \calL_0 U_{\lambda} \ra = \mu_{\lambda} \la  w, U_{\lambda} \ra\,,
  \end{equation}
  where we used in the marked equality that boundary terms vanish.

  Inequality \eqref{eq:projectiontruncated} is a direct consequence of
  linearity and norm-boundedness of the projections:
  \begin{align*}
    &\meanint_{W_R} \left|\Pi w_R-\Pi (w_R\,\eta)\right|^2
    = \meanint_{W_R} \left|\Pi (w_R (1-\eta)) \right|^2\\
    &\qquad \le \meanint_{W_R} \left|w_R (1-\eta)\right|^2 
    \le 
    \frac{C}{R^2} \int_{W_R\setminus W_{R-1}}
    \left|w_R\right|^2 \leq \frac{C}{R}\,.
  \end{align*}
  This concludes the proof.
\end{proof}

\paragraph{A warning concerning the non-periodicity of truncated
  solutions.}
Let $u$ be a vertically periodic solution of the Helmholtz equation on
$\R\times (0,h)$ and let $u^+_R$ be defined as in Definition \ref
{def:restrictionbox}.  Then $u^+_R$ is a solution of the Helmholtz
equation on the open square $W_R = R Y_\eps$. But $u^+_R$ is not a
{\em periodic solution} on the square (since it is not periodic in
horizontal direction).

This fact implies that certain formal calculations are wrong: Let
$u^+_R$ have the Bloch expansion $u^+_R = \sum_{\lambda\in I_R}
\alpha_\lambda U_\lambda^+$ for some coefficients $\alpha_\lambda$
(every $L^2$-function posesses such an expansion). Then, in general,
the following identity fails to hold:
$$\calL_0 \sum_{\lambda\in I_R} \alpha_\lambda U_\lambda^+
\stackrel{(?)}{=} \sum_{\lambda\in I_R} \alpha_\lambda
\calL_0(U_\lambda^+).$$

To see this, let us assume that the relation (?) holds. Then the
Helmholtz equation provides $\sum_{\lambda\in I_R} \omega^2
\alpha_\lambda U_\lambda^+ = \omega^2 u^+_R = \calL_0 u^+_R = \calL_0
\sum_{\lambda\in I_R} \alpha_\lambda U_\lambda^+ \stackrel{(?)}{=}
\sum_{\lambda\in I_R} \alpha_\lambda \calL_0(U_\lambda^+) =
\sum_{\lambda\in I_R} \alpha_\lambda \mu_\lambda^+ U_\lambda^+$.
Uniqueness of the Bloch expansion implies $\alpha_\lambda (\omega^2 -
\mu_\lambda^+) = 0$ for every $\lambda\in I_R$, hence $\alpha_\lambda
= 0$ for every $\lambda$ with $\mu_\lambda^+ \neq \omega^2$. We
conclude that the Bloch expansion of $u^+_R$ contains only
contributions from those basis functions $U_\lambda^+$ with
$\mu^+_\lambda = \omega^2$.

This is a contradiction for general solutions $u$: Let $\omega$ be a
frequency such that $\omega^2$ is not contained in the discrete set of
$(\mu^+_\lambda)_{\lambda\in I_R}$. Let furthermore $u$ be a
non-vanishing Bloch wave for the frequency $\omega$. Then the
expansion of $u$ does not use only $\alpha_\lambda = 0$, a
contradiction.

\begin{lemma}[Contributions from energy levels $m\ge 1$] 
  \label{lem:only-m=0} 
  Let $\omega$ satisfy the smallness condition \eqref
  {eq:smallnessfrequency} of Assumption \ref
  {ass:smallnessfrequency}. Let $u\in L^2_\loc(\R\times (0,h);\C)$ be
  a vertically periodic solution of the Helmholtz equation $\calL_0 u
  = \omega^2 u$ satisfying the uniform $L^2$-bounds of Definition
  \ref{def:outgoingwave}. Let $\eta = \eta_R$ be a family of cut-off
  functions as in \eqref{eq:cutofffct}. Then, with a constant $C$ that
  is independent of $R$:
  \begin{align}
    \label{eq:projection-m-ge-1}
    \meanint_{W_R} \left| \Pi^{\ev,\pm}_{m\ge 1} (u^\pm_R) \right|^2
    &\leq\frac{C}{R} \,\quad\text{ and }\quad \meanint_{W_R} \left|
      \Pi^{\ev,\pm}_{m\ge 1} (u^\pm_R\,\eta) \right|^2 \leq\frac{C}{R}\,.
  \end{align}
\end{lemma}

\begin{proof} 
  We perform the proof for the superscript ``+''.  Relation
  \eqref{eq:projectiontruncated} applied to $u^+_R$ provides
  \begin{align*}
    \meanint_{W_R} \left|\Pi^{\ev,+}_{m\ge 1} (u^+_R)-\Pi^{\ev,+}_{m\ge 1}
      (u^+_R\,\eta)\right|^2 \leq \frac{C}{R}\,.
  \end{align*}
  Indeed, by the uniform $L^2$-bounds of $u$, the condition
  $\tfrac{1}{R}\int_{W_R\setminus W_{R-1}}|u^+_R|^2\leq C_0$ with
  $C_0$ independent of $R$ is satisfied. The above inequality implies
  that it is sufficient to show only one of the two relations in
  \eqref{eq:projection-m-ge-1}, we show the second.

  \smallskip We now exploit Assumption \ref
  {ass:smallnessfrequency}. Due to \eqref {eq:smallnessfrequency},
  there exists $\delta > 0$ such that $| \omega^2 - \mu_{\lambda}|^2
  \ge \delta$ for all $\lambda=(j,m)$ with $m\geq 1$.  We therefore
  find
  \begin{align*}
    &\delta\meanint_{W_R}\left|\Pi^{\ev,+}_{m\ge 1} (
      u^+_R\,\eta)\right|^2 =\delta\sum_{\lambda=(j,m)\in I_R\atop
      m\geq 1}\left|\left\langle u_R^+\,\eta,
        U_\lambda\right\rangle_R\right|^2 \leq\sum_{\lambda=(j,m)\in
      I_R\atop m\geq 1}
    \left|(\omega^2-\mu_\lambda)\langle u_R^+\,\eta, U_\lambda\rangle_R\right|^2\displaybreak[2]\\
    &\leq \sum_{\lambda\in I_R}\left|\left\langle \omega^2
        u_R^+\,\eta, U_\lambda\right\rangle_R -\left\langle
        \mu_\lambda u_R^+\,\eta, U_\lambda\right\rangle_R \right|^2 =
    \sum_{\lambda\in I_R}\left| \left\langle \calL_0(u_R^+)\,\eta,
        U_\lambda\right\rangle_R - \left\langle \mu_\lambda
        u_R^+\,\eta, U_\lambda\right\rangle_R
    \right|^2\displaybreak[2]\\
    & = \meanint_{W_R}\left|\calL_0 (u_R^+)\eta -\sum_{\lambda\in I_R}
      \mu_\lambda\langle u^+_R\eta, U_\lambda\rangle_R\, U_\lambda
    \right|^2\leq\frac{C}{R}\,.
  \end{align*}
  In the second line we used that $\calL_0(u_R^+)=\omega^2 u_R^+$
  holds pointwise almost everywhere in $W_R$. In the last inequality
  we used \eqref{eq:solutions-and-mu-factors-weight}, exploiting the
  uniform $H^1$-bounds provided by Lemma \ref{lem:Caccioppoli}. This
  concludes the proof.
\end{proof}

\subsection{Other radiation conditions and the sesquilinear form
  $b^\pm_R$}

In this section we discuss how our radiation condition simplifies in
the case of a homogeneous medium.  We furthermore show that the
radiation condition suggested by Fliss and Joly in \cite
{FlissJoly2016} is formally stronger than our condition: Every
solution that satisfies the condition of \cite {FlissJoly2016}
satisfies also our condition. Finally, we introduce the sesquilinear
form $b^\pm_R$, which plays a major role in our proofs. The form
$b^\pm_R$ can also be used to introduce an even weaker form of the
outgoing wave condition.

\subsubsection*{In free space, Bloch expansions are Fourier
  expansions} Let us study the outgoing wave condition in a
homogeneous medium. This is our situation for $x_1<0$, indicated with
the superscript ``-''.  In a homogeneous medium, the Bloch waves are
harmonic waves. This allows to give explicit formulas for some
quantities and, in particular, for the outgoing wave condition.

\begin{remark}[Basis functions and Poynting numbers in a homogeneous medium]
  \label{rem:Left-Fourier}
  The functions $\Psi^-_{j,m}$ are $\eps$-periodic eigenfunctions of
  the operator $\calL_j^{-} = -\left(\nabla + 2\pi ij/\eps\right)
  \cdot \left(\nabla + 2\pi ij/\eps\right)$. As such, for fixed $j\in
  Z$, they are harmonic waves, $\{\Psi^-_{j,m}\,|\, m\in\N_0\} =
  \{e^{2\pi i k\cdot x/\eps}\,|\,k\in\Z^2\}$.  More precisely, for
  every $j\in Z$ and every $m\in\N$, there exists a wave-number $k =
  k(j,m)\in\Z^2$ such that
  \begin{equation}
    \label{eq:Fourierbasis}
    \Psi^-_{j,m}(x)=e^{2\pi i k\cdot x/\eps}\,,\qquad \mu^{-}_m(j)
    = 4\pi^2 \frac{|k+j|^2}{\eps^2}\,.
  \end{equation}
  Accordingly, for $\lambda=(j,m)$, we have $U^-_{\lambda}(x) =
  e^{2\pi i\left(k+j\right)\cdot x/\eps}$, and the Poynting number is
  \begin{equation*}
    P^-_\lambda = \Im\, \meanint_{Y_\eps} \bar U^-_\lambda (x)
    e_1\cdot \nabla U^-_\lambda(x)\, dx =\frac{2\pi}{\eps}\left(k_1
      + j_1\right)\,.
  \end{equation*}
  In particular, for the first energy level, $\lambda=(j,0)$, we find
  \begin{equation*}
    k=k(j,0)\in\argmin_{k\in\Z^2}{|k+j|^2}\,,
  \end{equation*}
  and thus, for $j=(j_1,j_2)$ with $j_1\neq \frac12$:
  \begin{equation}
    \label{eq:Poyntingm0}
    \frac{\eps}{2\pi}\ P^-_{(j,0)}  = j_1 + \argmin_{k_1\in\Z}{|k_1+j_1|^2}
    =
    \begin{cases}
      j_1&\quad\text{for }j_1\in[0,\frac12)\\
      j_1-1& \quad\text{for }j_1\in (\frac12,1]
    \end{cases}.
  \end{equation}
  The wave $U^-_{(j,0)}$ is right-going in the sense of
  Definition \ref{def:Projections} if and only if $j_1\in
  [0,\tfrac12)$.
\end{remark}

\begin{remark}[Outgoing wave condition in a homogeneous medium]
  \label{rem:outgoinghomogeneous}
  Let $u_R^-$ be as in Definition \ref{def:restrictionbox}.  Remark
  \ref{rem:Left-Fourier} implies that, for every index $\lambda =
  (j,m)\in I_R$, there exists an index $k=k(\lambda)\in\Z^2$ with
  $U^-_{\lambda}(x) = e^{2\pi i\left(k+j\right)\cdot x/\eps} = e^{2\pi
    iR\left(k+j\right)\cdot x/R\eps}$.  Using the shorthand notation
  $l(\lambda) := R\left(k(\lambda) + j\right) \in\Z^2$, the Bloch
  expansion of $u_R^-$ can be rewritten as a Fourier expansion,
  \begin{equation*}
    u_R^-(x)=\sum_{\lambda\in I_R}\alpha^-_{\lambda,R}U^-_\lambda(x)
    =\sum_{\lambda\in I_R}\alpha^-_{\lambda,R}\, e^{2\pi i
      l(\lambda)\cdot x / R\eps}\,.
  \end{equation*}
  By Remark \ref{rem:Left-Fourier}, the basis function
  $U_\lambda^-(x)$ is right-going in the sense of Definition
  \ref{def:Projections} if and only if $l_1(\lambda) > 0$. This
  simplifies the radiation condition: The function $u$ with the
  truncations $u_R^-$ on the left satisfies the outgoing wave
  condition \eqref{eq:outgoingleft} if and only if
  \begin{equation}
    \label{eq:outgoingfourier1}
    \sum_{\lambda\in I_R\atop l_1(\lambda)>0}|\alpha^-_{\lambda,R}|^2
    \rightarrow  0\quad\text{as }R\rightarrow \infty\,.
  \end{equation}
\end{remark}

We emphasize that, in order to evaluate the radiation condition, the
map $l : I_R\to\Z^2$ need not be evaluated (we know that it is
bijective). By expanding $u_R^-$ in a classical Fourier series,
$u_R^-(x)=\sum_{l\in\Z^2}\beta^-_{l,R}\, e^{2\pi i l\cdot x /R\eps}$
with some coefficients $\beta^-_{l,R}$, the outgoing wave condition
\eqref{eq:outgoingfourier1} is equivalent to $\sum_{l\in\Z^2\atop
  l_1>0}|\beta^-_{l,R}|^2 \rightarrow 0$ as $R\rightarrow \infty$.

\subsubsection*{Comparison to the outgoing wave condition of Fliss and
  Joly \cite{FlissJoly2016}}

We claim that the outgoing wave condition of Fliss and Joly is
formally stronger than our condition. More precisely: Every solution
$u$ that satisfies the radiation condition \eqref {eq:FlissJolyCond}
of \cite{FlissJoly2016} satisfies also our radiation condition
\eqref{eq:outgoingright} (we restrict the considerations here to the
right side $x_1\rightarrow+\infty$).  Indeed, let $u$ be as in \eqref
{eq:FlissJolyCond}, i.e.\,a finite sum of right-going Bloch-waves plus
an exponentially decaying remainder $w^+(x)$.  In order to check
\eqref{eq:outgoingright} for $u$, it suffices to verify, for each of
the finitely many terms, the smallness of its
$\Pi^+_{<0}$-projection. The smallness of the projection of $w^+(x)$
is clear because of the exponential decay of $w^+(x)$ and the
boundedness of the projection operator. The smallness of the
projection for each of the Bloch waves is shown in the subsequent
lemma.

Since the expansion \eqref {eq:FlissJolyCond} contains only vertically
periodic waves with frequency $\omega$ that are outgoing, we restrict
our analysis to such Bloch-waves $U^+_\lambda$.

\begin{lemma}[Right-going Bloch waves satisfy the radiation condition]
  \label{lem:OutgoingwaveSingleWave}
  Let $K\in\N$ denote the number of periodicity cells in vertical
  direction and let $U^+_\lambda$ be a Bloch wave with $\lambda =
  (j,m) \in Z\times\N_0$, $j_2\in Q_K
  =\{0,\frac{1}{K},...,\frac{K-1}{K}\}$, $\mu_m^+(j)=\omega^2$, and
  $P^+_\lambda>0$.  We impose that the frequency $\omega$ satisfies
  Assumption \ref{ass:smallnessfrequency} (hence $m=0$).  Then, as $\N
  K\ni R\to \infty$:
  \begin{align}
    \label{eq:projectionsinglewave1}
    \meanint_{RY_\eps}\left|\Pi^+_{< 0}((U^+_\lambda)^+_R)\right|^2
    &\rightarrow 0\,,\\
    \label{eq:projectionsinglewave2}
    \meanint_{RY_\eps}\left|\Pi^+_{\geq
        0}((U^+_\lambda)^+_R)\right|^2&\rightarrow 1\,.
  \end{align}
\end{lemma}

\begin{proof} {\it Step 1: Equivalence of \eqref
    {eq:projectionsinglewave1} and \eqref {eq:projectionsinglewave2}.}
  By $L^2$-orthogonality of the projections we can calculate
  \begin{equation*}
    1=\meanint_{RY_\eps}|(U^+_\lambda)^+_R|^2
    =\meanint_{RY_\eps}\left|\Pi^+_{\geq 0}((U^+_\lambda)^+_R)\right|^2 + 
    \meanint_{RY_\eps}\left|\Pi^+_{< 0}((U^+_\lambda)^+_R)\right|^2\,,
  \end{equation*}
  which yields the equivalence of \eqref{eq:projectionsinglewave1} and
  \eqref{eq:projectionsinglewave2}.

  {\it Step 2: Proof of \eqref{eq:projectionsinglewave1}.}  Let
  $\lambda = (j,m) = (j,0)$ be as in the lemma. Arguing exactly as in
  Lemma \ref{lem:only-m=0} we conclude that contributions from energy
  levels $m\geq 1$ are negligible,
  \begin{equation*}
    \meanint_{RY_\eps}\left|\Pi^{\ev,+}_{m\geq 1}((U_\lambda^+)^+_R)\right|^2\rightarrow 0
    \quad\text{as }R\rightarrow\infty. 
  \end{equation*}
  Consequently, for the weighted $L^2$-norm of $\Pi^+_{<
    0}((U_\lambda^+)^+_R)$ we calculate
  \begin{align*}
    \meanint_{RY_\eps}\left|\Pi^+_{< 0}((U^+_\lambda)^+_R)\right|^2
    &=\sum_{\tilde\lambda\in I_R\cap I^+_{<0}}\left|\langle
      (U_\lambda^+)^+_R, U_{\tilde\lambda}\rangle_R\right|^2=
    \sum_{\tilde\lambda\in I_R\cap I^+_{<0}\atop \tilde
      m=0}\left|\langle (U_\lambda^+)^+_R,
      U_{\tilde\lambda}\rangle_R\right|^2+ o(1)\\
    &=\sum_{\tilde\lambda\in I_R\cap I^+_{<0}\atop \tilde m=0,\,
      \tilde j_2=j_2} \left|\langle (U_\lambda^+)^+_R,
      U_{\tilde\lambda}\rangle_R\right|^2+ o(1)\quad\text{as
    }R\rightarrow\infty.
  \end{align*}
  In the last line we exploited that due to $j_2\in Q_R$ all scalar
  products with $\tilde j_2\neq j_2$ vanish.  Next we show that there
  exists a constant $C=C(\lambda)>0$ such that
  \begin{equation}
    \label{eq:estimateBlochcoefficients}
    \left|\langle (U_\lambda^+)^+_R, U_{\tilde\lambda}\rangle_R\right|\leq \frac{C}{R}
  \end{equation}
  for all $\tilde\lambda=(\tilde j,0)\in I_R\cap I^+_{<0}$ with
  $\tilde j_2=j_2$. Indeed, a direct calculation analogous to that of
  Lemma \ref{lem:orthweight} yields
  \begin{align*}
    \langle (U_\lambda^+)^+_R, U_{\tilde\lambda}\rangle_R=e^{-2\pi i
      j_1 R}\,\frac{C(\lambda,\tilde\lambda)}{R}\frac{1- e^{2\pi
        i(\tilde j_1-j_1)R}}{1-e^{2\pi i(\tilde j_1-j_1)}}
  \end{align*} 
  with
  $C(\lambda,\tilde\lambda):=\meanint_{(0,\eps)^2}\Psi^+_\lambda\bar{\Psi}^+_{\tilde\lambda}e^{2\pi
    i (\tilde j_1-j_1)y_1/\eps}dy$\,.  In particular
  $|C(\lambda,\tilde\lambda)|\leq 1$ and therefore
  \begin{align}
    \label{eq:blochcoefficientestimate}
    \left|\langle (U_\lambda^+)^+_R,
      U_{\tilde\lambda}\rangle_R\right|\leq \frac{1}{R}\left|\frac{1-
        e^{2\pi i(\tilde j_1-j_1)R}}{1-e^{2\pi i(\tilde
          j_1-j_1)}}\right| \leq\frac{1}{R}\frac{2}{\left|1-e^{2\pi
          i(\tilde j_1-j_1)}\right|}\,.
  \end{align}
  We now exploit $P^+_\lambda>0$. The eigenvalue $\mu^+_0(j)$ is
  simple, hence the Poynting number $P^+_\lambda=P^+_{(j,0)}$ is
  continuous in the wave number $j$. One thus finds a positive
  constant $\delta=\delta(\lambda)>0$ such that
  \begin{align*}
    \left|1-e^{2\pi i(\tilde j_1-j_1)}\right|> \delta
  \end{align*} 
  for all $\tilde\lambda=(\tilde j,0)\in I_R\cap I^+_{<0}$ with
  $\tilde j_2=j_2$.  Together with \eqref{eq:blochcoefficientestimate}
  this yields the claim \eqref{eq:estimateBlochcoefficients} with
  $C=\frac{2}{\delta}$.

  With estimate \eqref{eq:estimateBlochcoefficients} at hand we
  conclude
  \begin{align*}
    \meanint_{RY_\eps}\left|\Pi^+_{< 0}((U^+_\lambda)^+_R)\right|^2
    &\leq \frac{C^2}{R^2}\left|\left\{\tilde\lambda=(\tilde j,0)\in
        I_R\cap I^+_{<0}
        \text{ with } \tilde j_2=j_2\right\}\right|+o(1)\\
    &\leq\frac{C^2}{R}+ o(1)\rightarrow 0\quad\text{as }R\rightarrow
    \infty\,,
  \end{align*}
  which was the claim.
\end{proof}

\subsubsection*{The sesquilinear form $b^\pm_R$}

Let $\eta = \eta_R$ be a family of cut-off functions as in \eqref
{eq:cutofffct}.  In view of Relation \eqref {eq:projectiontruncated}
of Lemma \ref{lem:almost-psi}, the outgoing wave conditions
\eqref{eq:outgoingright} and \eqref{eq:outgoingleft} are equivalent to
outgoing wave conditions for the truncated functions $u^\pm_{R,\eta}
:= u^\pm_R\,\eta$. More precisely, they are equivalent to the
conditions
\begin{equation}
  \label{eq:outgoingtruncated}
  \meanint_{W_R}\left|\Pi^+_{<0}(u^+_{R,\eta})\right|^2\rightarrow 0\quad
  \text{ and } \quad
  \meanint_{W_R}\left|(\Pi^-_{>0}(u^-_{R,\eta})\right|^2\rightarrow 0\quad
  \text{ as }\quad  R\to \infty\,.
\end{equation}
In the proof of our uniqueness result we will use \eqref
{eq:outgoingtruncated} instead of the original conditions
\eqref{eq:outgoingright} and \eqref{eq:outgoingleft}. In fact, even a
weaker form of the conditions is sufficient and we discuss this
relaxation in the following.

Corresponding to the energy flux definition in \eqref
{eq:Poynting-def}, we associate to a function $w\in H^1(W_R;\C)$ on
$W_R = R Y_\eps$ the Poynting number
\begin{equation}
  \label{eq:Poynting-def-with-R}
  B_R^+(w) := \Im\, \meanint_{W_R} \bar w(x)
  e_1\cdot \left[a^\eps(x)\nabla w(x)\right]\, dx\,.
\end{equation}
The quadratic expression $B_R^-$ is defined analogously,
with $a^\eps(x)$ replaced by $1$.

\begin{definition}[Weaker form of the outgoing wave condition]
  \label{def:weakoutgoingwave}
  For $K, R\in\N$ with $R\in K \N$ we consider $u\in H^1_\loc(\R\times
  (0,\eps K);\C)$ and $u^\pm_{R,\eta}$ as in \eqref
  {eq:outgoingtruncated}. We say that $u$ satisfies the energetic
  outgoing wave condition on the right, if
    \begin{equation}
      \label{eq:weakoutgoingright}
      B_R^+\left(\Pi^+_{<0}\Pi^{{\ev},+}_{m=0}(u^+_{R,\eta})\right)\rightarrow 0
      \text{ as }  R\to \infty\,.
    \end{equation}
    Accordingly, we say that $u$ satisfies the energetic outgoing wave
    condition on the left, if
    \begin{equation}
      \label{eq:weakoutgoingleft}
       B_R^-\left(\Pi^-_{>0}\Pi^{{\ev},-}_{m=0}(u^-_{R,\eta})\right)\to 0
      \text{ as } R\to \infty\,.
    \end{equation}
\end{definition}

In two respects, the condition \eqref {eq:weakoutgoingright} is
similar to the condition \eqref{eq:outgoingtruncated}: the function
$u$ is considered at the far right since only $u^+_{R,\eta}$ is
used. In view of Lemma \ref{lem:only-m=0}, contributions from energy
levels $m\geq 1$ can be neglected and we consider only
$\Pi^{{\ev},+}_{m=0}(u^+_{R,\eta})$.  Furthermore, this function
is projected to left-going waves, i.e.\,only
$\Pi^+_{<0}\Pi^{{\ev},+}_{m=0}(u^+_{R,\eta})$ is studied. The
main difference between the two conditions is that, instead of looking
at the weighted $L^2$-norm, one demands in \eqref
{eq:weakoutgoingright} a decay property for the energy-flux quantity
$B_R^+$.  At the end of this section, we will see that condition
\eqref{eq:outgoingtruncated} (together with the uniform $L^2$-bounds
and the solution property) implies \eqref{eq:weakoutgoingright}.

The definition of $B_R^+$ in \eqref {eq:Poynting-def-with-R} suggests
to introduce additionally the (nonsymmetric) sesquilinear forms
$b_R^\pm : L^2(W_R;\C) \times H^1(W_R;\C)\rightarrow\C$,
\begin{equation}
  \label{eq:bpm-def}
  \begin{split}
    b^+_R(u,v)&:=\meanint_{W_R}\bar u(x)\,
    e_1\cdot\left[a^\eps(x)\nabla v(x)\right]\,dx\,,\\
    b^-_R(u,v)&:=\meanint_{W_R}\bar u(x)\,e_1\cdot\nabla v(x)\,dx\,.
  \end{split}
\end{equation}
The definition is tailored to calculate energy fluxes. The energy flux
of the left-going contributions of $u^+_{R,\eta}$ (in the right
half-plane) is quantified by
\begin{align*}
  B_R^+(\Pi^+_{<0} u^+_{R,\eta}) &= \Im\,
  b^+_R\left(\Pi^+_{<0}u^+_{R,\eta},\Pi^+_{<0}u^+_{R,\eta}\right) \\
  &= \Im\,
  \meanint_{W_R}\overline{\Pi^+_{<0}u^+_{R,\eta}(x)}e_1\cdot
  \left[a^\eps(x)\nabla(\Pi^+_{<0}u^+_{R,\eta})(x)\right]\,dx\,.
\end{align*}
The connection to $ P^\pm_\lambda$ is given by
\begin{align}
  \label{eq:relationbS-2}
  P^\pm_\lambda 
  = B_R^\pm(U^\pm_\lambda)
  = \Im\, b_R^\pm\left(U^\pm_\lambda, U^\pm_\lambda\right)\,.
\end{align}

Let us collect some properties of the sesquilinear forms $b_R^\pm$.

\begin{lemma}[Properties of the sesquilinear form $b_R^\pm$]
  \label{lem:orthogonalwavenumber}
  For $R\in\N$, the following holds:

  {\em 1.  Orthogonality property of $b^\pm_R$.}  Let $\lambda,
  \tilde\lambda\in I_R$ be such that $\lambda=(j,m)$,
  $\tilde\lambda=(\tilde j,\tilde m)$ with $j\neq \tilde j$.  Then
  $U^\pm_\lambda,U^\pm_{\tilde\lambda}$ of
  \eqref{eq:abbreviationUlambda} satisfy
  \begin{align}
    \label{eq:orthogonalwavenumber}
    b^\pm_R(U_\lambda^\pm, U_{\tilde\lambda}^\pm) = 0 \,.
  \end{align}

  {\em 2. Convergence property of $b^\pm_R$.}  Let sequences
  $u_R\in L^2(W_R;\C)$ and $v_R\in H^1(W_R;\C)$ be such that
  \begin{align}
    \label{eq:uniformbounds1}
    \meanint_{W_R}|u_R|^2 + |\nabla v_R|^2\leq C_0
  \end{align}
  with $C_0$ independent of $R$. Let either
  $\meanint_{W_R}|u_R|^2\rightarrow 0$ or $\meanint_{W_R}|\nabla
  v_R|^2\rightarrow 0$ as $R\rightarrow \infty$.  Then there holds
  \begin{equation}
    b^\pm_R(u_R,v_R)\to 0\,.
  \end{equation}
\end{lemma}

\begin{proof}
  1. We prove \eqref{eq:orthogonalwavenumber} for
  $U^+_\lambda,U^+_{\tilde\lambda}$, the argument for $U^-_\lambda,
  U^-_{\tilde\lambda}$ is analogous. We have to show that
  \begin{align*}
    b_R^+(U_\lambda^+, U^+_{\tilde\lambda})
    =\meanint_{W_R}\overline{U_\lambda^+}(x)\,e_1
    \cdot\left[a^\eps(x)\nabla U^+_{\tilde\lambda}(x)\right]\,dx
    \stackrel{!}{=} 0\,.
  \end{align*}
  By definition of $U^+_\lambda$ and $U^+_{\tilde\lambda}$ there holds
  \begin{align*}
    \overline{U^+_\lambda}(x)&=\overline{\Psi^+_\lambda}(x)e^{-i 2\pi j\cdot x/\eps},\\
    \nabla
    U^+_{\tilde\lambda}(x)&=\left[\nabla\Psi^+_{\tilde\lambda}(x) +
      (i 2\pi \tilde j/\eps)\, \Psi^+_{\tilde\lambda}(x)\right] e^{i
      2\pi\tilde j\cdot x/\eps}
  \end{align*}
  with $\eps$-periodic functions $\Psi^+_\lambda$,
  $\Psi^+_{\tilde\lambda}$, and $\nabla \Psi^+_{\tilde\lambda}$.  Due
  to the $\eps$-periodicity of $a^\eps$ and since $j, \tilde j\in Q_R$
  satisfy $j\neq \tilde j$, we can apply Lemma \ref{lem:orthweight} of
  the appendix, which yields the claim.

  \smallskip 2. We show the claim for $b_R^+$, the argument for
  $b_R^-$ is analogous. The Cauchy-Schwarz inequality allows to
  calculate
  \begin{align*}
    &\left|b^+_R(u_R,v_R)\right|
    =\left|\meanint_{W_R}\bar u_R(x)\,e_1\cdot
      \left[a^\eps(x)\nabla v_R(x)\right]\,dx\right|\\
    &\quad\leq\|a^\eps\|_{\infty}\left(\meanint_{W_R}|u_R|^2\right)^{1/2}
    \left(\meanint_{W_R}\left|\nabla
        v_R\right|^2\right)^{1/2} \rightarrow 0\quad\text{as
    }R\rightarrow\infty\,,
  \end{align*}
  which concludes the proof.
\end{proof}

Lemma \ref{lem:orthogonalwavenumber} shows that the outgoing wave
condition \eqref{eq:outgoingtruncated} together with the $L^2$-bounds
of Definition \ref{def:outgoingwave} imply \eqref
{eq:weakoutgoingright}. Indeed, by \eqref{eq:outgoingtruncated}, there
holds
\begin{align*}
  \meanint_{W_R}\left|\Pi^+_{<0}\Pi^{{\ev},+}_{m=0}(u^+_{R,\eta})\right|^2
  \leq\meanint_{W_R}\left|\Pi^+_{<0}(u^+_{R,\eta})\right|^2\rightarrow 0
  \text{ as } R\to \infty\,.
\end{align*}
Moreover, $\Pi^+_{<0}\Pi^{{\ev},+}_{m=0}(u^+_{R,\eta})$ satisfies
$\meanint_{W_R}|\nabla
\left(\Pi^+_{<0}\Pi^{{\ev},+}_{m=0}(u^+_{R,\eta})\right)|^2\leq C$
with $C$ independent of $R$ due to Lemma \ref{lem:Caccioppoli} and
Lemma \ref{lem:RegularityProjectionsev} of the appendix. Lemma
\ref{lem:orthogonalwavenumber} provides
\begin{align*}
  B_R^+\left(\Pi^+_{<0}\Pi^{{\ev},+}_{m=0}(u^+_{R,\eta})\right)
  =\Im\,b^+_R
  \left(\Pi^+_{<0}\Pi^{{\ev},+}_{m=0}(u^+_{R,\eta}),
    \Pi^+_{<0}\Pi^{{\ev},+}_{m=0}(u^+_{R,\eta})\right)
  \rightarrow 0\text{ as } R\to \infty\,,
\end{align*}
and hence \eqref{eq:weakoutgoingright}.

\begin{remark}[On the sesquilinear form $b_R^\pm$]
  Another choice of a bilinear form is
  \begin{equation}
    \label{eq:b-alt}
    \tilde b^+_R(u,v) := \frac12 \meanint_{W_R}
    \left\{ \bar u(x)\, e_1\cdot\left[a^\eps(x)\nabla v(x)\right]
      \phantom{I^{I^I}}\hspace*{-6mm} - 
      v(x)\, e_1\cdot\left[a^\eps(x)\nabla \bar u(x)\right] \right\}
      \,dx\,.    
  \end{equation}
  With this choice, the energy flux $B^\pm_R$ can be calculated as
  before, since $\Im\, b^+_R(u,u) = \Im\, \tilde b^+_R(u,u)$ holds for
  every $u$. The properties of Lemma \ref {lem:orthogonalwavenumber}
  remain also unchanged, the only additional requirement would be an
  $H^1$-bound also for $u_R$ in \eqref {eq:uniformbounds1}.

  The advantage of $\tilde b^+_R(u,v)$ is that more orthogonality can
  be expected for $\tilde b^+_R$ than for $b^+_R$. Essentially, the
  bilinear form $q$ of (27) in \cite{FlissJoly2016} coincides with
  $\tilde b^+_R$ (up to a factor $2$, our coefficient $a^\eps$, and
  the fact that we use a volume integral for the averaging). In
  Theorem 3 of \cite{FlissJoly2016}, an orthogonality property is
  shown for $q$, which resembles our orthogonality relation \eqref
  {eq:orthogonalwavenumber}, stating that orthogonality holds also for
  $\lambda=(j,m)$ and $\tilde\lambda=(j,\tilde m)$ with $m\neq \tilde
  m$. Unfortunately, such an orthogonality is only true for
  basis functions corresponding to the same frequency $\omega$, while
  our analysis of an interface would require orthogonality independent
  of the frequency.
\end{remark}

\section{Bloch measures and uniqueness properties}
\label{sec.uniqueness}

Our aim is to show uniqueness properties of the transmision Problem
\ref{prob:Transmission} with incoming wave $U_\inc$ and outgoing wave
conditions. Following the standard procedure of uniqueness proofs, we
consider two solutions $u$ and $\tilde u$ of the problem. Due to
linearity of the system, the difference $v := u - \tilde u$ satisfies
again \eqref{eq:Peps}. Furthermore, it satisfies outgoing wave
conditions on the left and on the right according to Definition
\ref{def:outgoingwave}, without any incoming wave $U_\inc$. At this
point, we have exploited the triangle inequality: Certain projections
of $u$ and $\tilde u$ tend to zero in a weighted $L^2$-norm, hence
also the projections of $v$ tend to zero. We can not show that $v$
vanishes (indeed, as explained in the introduction, we expect that
there exist nontrivial solutions for vanishing $U_\inc$). But we can
show that the functions $v^\pm_R$ consist, in the limit
$R\rightarrow\infty$, only of vertical waves.  The right object to
study is the {\it Bloch measure} associated with $v^\pm_R$.

We recall that the frequency assumption \eqref{eq:smallnessfrequency}
implies that, in the limit $R\rightarrow\infty$, the discrete Bloch
expansions of $u^\pm_R$ contain only modes corresponding to $\lambda
= (j,m)$ with $m=0$, see Lemma \ref{lem:only-m=0}.

\subsection{Bloch measures}

In the definition of Bloch measures we use the space of all Radon
measures on the unit square $Z$, which we denote as $\calM(Z)$. It is
the dual of the space of continuous functions on $Z$ and accordingly
equipped with the topology of weak-$*$ convergence. This means, in
particular, that every bounded sequence in $\calM(Z)$ has a convergent subsequence.

\begin{definition}[Discrete Bloch measure]
  \label{def:Blochmeasure}
  Let $u_R\in L^2(W_R;\C)$ be a sequence of functions with discrete
  Bloch-expansions
  \begin{align*}
    u_R(x)=\sum_{\lambda\in I_R} \alpha^\pm_\lambda U^\pm_\lambda(x)\,,
  \end{align*}
  where $\alpha^\pm_\lambda = \alpha^\pm_\lambda(R)$ depend on $R\in
  \N$.  Given these coefficients, for fixed $l\in\N_0$, we define the
  $l$-th discrete Bloch-measure $\nu_{l,R}^\pm \in \calM(Z)$ by
  \begin{align}
	 \label{eq:Blochmeasure1}
    \nu_{l,R}^\pm := \sum_{\lambda = (j,l) \in I_R}
    |\alpha_\lambda^\pm|^2\,\ \delta_{j}\,,
  \end{align}
  where $\delta_{j}$ denotes the Dirac measure at the frequency $j\in
  Z$.
\end{definition}

For $u_R$ fixed, $\nu_{l,R}^\pm$ is a non-negative Radon measure on $Z
= [0,1]^2$.  There holds
\begin{align}
  \label{eq:Blochmeasure2}
  \sum_{l=0}^\infty \int_{Z}\, d\nu_{l,R}^\pm =\sum_{\lambda\in I_R}
  |\alpha_\lambda^\pm|^2 =\meanint_{W_R} |u_R|^2\,.
\end{align}
Our aim is to study the limiting behavior $R\rightarrow \infty$ of the
discrete Bloch measures $\nu_{l,R}^\pm$.

\begin{definition}[Bloch measure]
  \label{def:limitingBlochmeasure}
  For $\eps>0$, $K\in\N$ and $h = K\eps$, let $u$ be a function $u \in
  L^2_\loc(\R\times (0,h);\C)$. We consider a sequence $\N K \ni R\to
  \infty$. We extract $u^\pm_{R,\eta}:=u^\pm_R\,\eta$ according to
  Definition \ref {def:restrictionbox} with a sequence of cut-off
  functions $\eta=\eta_R$ as in \eqref {eq:cutofffct}. For $l\in\N_0$,
  let $\nu^\pm_{l,R}$ be the discrete Bloch measures associated with
  $u^\pm_{R,\eta}$.

  We say that the measure $\nu^\pm_{l,\infty}\in \calM(Z)$ is a Bloch
  measure generated by $u$ if there holds, along a subsequence
  $R\to\infty$, in the sense of measures (i.e.\,weak-$*$),
  \begin{equation}
    \label{eq:Bloch-measure-conv}
    \nu_{l,R}^\pm \to \nu^\pm_{l,\infty}\,.
  \end{equation}
\end{definition}

Relation \eqref {eq:Bloch-measure-conv} is equivalent to the
following: for every test-function $\phi\in C(Z)$ on $Z=[0,1]^2$ there
holds
\begin{equation*}
  \sum_{\lambda = (j,l) \in I_R} \phi(j) |\alpha_\lambda^\pm|^2
  =  \int_{Z}\phi\,d\nu_{l,R}^\pm \to \int_{Z}\phi\,d\nu^\pm_{l,\infty}
  \quad\text{as } R\to\infty\,.
\end{equation*}

The methods of Section \ref{ssec.Bloch-uniqueness} force us to work
with cut-off functions, i.e.\,with $u^\pm_{R,\eta}$ and with
$\nu^\pm_{l,R}$. Nevertheless, one may also study the discrete Bloch
measures $\tilde\nu^\pm_{l,R}$ associated to $u^\pm_R$ (without
cut-off function). In the limit $R\rightarrow \infty$, the two
measures coincide,
\begin{equation*}
  \nu^\pm_{l,R} - \tilde\nu^\pm_{l,R}\to 0
\end{equation*}
in the sense of measures.  In particular, the Bloch measure generated
by $u^\pm_{R,\eta}$ is independent of the choice of $\eta$.

\subsection{Energy estimates and consequences for the Bloch measure}

Up to this point (with the exception of Lemma \ref {lem:only-m=0}),
our considerations have been completely abstract in the following
sense: Given a function $u \in L^2_\loc(\R\times (0,h);\C)$, we have
constructed restrictions of $u$ to large boxes, projections of these
restrictions, and finally discrete and limiting Bloch measures
corresponding to $u$. Except for regularity properties, we have not
exploited the Helmholtz equation. In this section, we will derive
relations that express a physical law: energy conservation. This will
eventually lead us to the uniqueness properties which are expressed
with the Bloch measures.

The subsequent result states that, while left-going waves on the right
vanish by the outgoing wave condition, right-going waves vanish by
energy conservation.

\begin{proposition}
  \label{prop:weightedsumalpha}
  Let Assumption \ref{ass:smallnessfrequency} on $\omega>0$ be
  satisfied and let $v$ be a solution to the scattering problem
  \eqref{eq:Peps}, periodic in vertical direction, satisfying outgoing
  wave conditions on the left and on the right according to Definition
  \ref{def:outgoingwave}, without incoming wave, i.e.\, $U_\inc \equiv
  0$. For a sequence of cut-off functions $\eta=\eta_R$ as in \eqref
  {eq:cutofffct} we consider $v_{R,\eta}^\pm := v_{R}^\pm\,\eta_R =
  \sum_{\lambda\in I_R} \alpha_{\lambda,R}^\pm U_\lambda^\pm$,
  c.f.\,Definition \ref{def:restrictionbox}. Then, as $\N K\ni R\to
  \infty$,
  \begin{align}\label{eq:result-P-weighted-sum}
    \sum_{\lambda=(j,0)\atop\lambda\in I_{R}\cap I^-_{\leq 0}}
    |\alpha^-_{\lambda,R}|^2\   P^-_\lambda\to 0\quad\text{and}\quad
    \sum_{\lambda=(j,0)\atop\lambda\in I_{R}\cap I^+_{\ge 0}}|\alpha^+_{\lambda,R}|^2\ 
     P^+_\lambda  \to 0\,.
  \end{align}
\end{proposition}

\begin{proof}
  {\em Step 1: Energy flux equality.}  For $h = \eps K$ and $R \in \N
  K$, we consider the special cut-off function $\vartheta(x) =
  \vartheta_R(x)$, defined for $x = (x_1, x_2)$ as
  \begin{align*}
    \vartheta(x) :=
    \begin{cases}
        1 & \text{ if } |x_1|\leq \eps R\,,\\
        2-\tfrac{|x_1|}{\eps R}& \text{ if } \eps R < |x_1| < 2\eps R\,,\\
        0 & \text{ if } |x_1|\geq 2\eps R\,.
    \end{cases}
  \end{align*}
  We multiply the Helmholtz equation \eqref{eq:Peps} with coefficients
  $a = a^\eps$ and solution $v$ by the test-function $\vartheta(x)
  \,\overline{v}(x)$. An integration over $\R\times(0,h)$ and
  integration by parts yields (no boundary terms appear due to
  periodicity in $x_2$-direction and compact support):
  \begin{align*}
    \int_\R\int_0^h 
    \left\{ a^\eps\, \vartheta\, \left|\nabla v\right|^2 
      + a^\eps\, \del_{x_1} \vartheta\ \overline{v}\ \del_{x_1} v
    \right\}
    = \omega^2\int_\R\int_0^h \vartheta\, |v|^2\,.
  \end{align*}
  Due to the special choice of $\vartheta$ and $a^\eps(x) = 1$ for
  $x_1<0$, this equation reads
  \begin{equation*}
    \meanint_{-2R\eps}^{-R\eps}\int_0^h\overline{v}\, \del_{x_1}v
    -\meanint_{R\eps}^{2R\eps}\int_0^h\overline{v}\,a^\eps\del_{x_1}v
    = \int_\R\int_0^h 
    \left\{ \omega^2\vartheta\left|v\right|^2
      - a^\eps\,\vartheta\,\left|\nabla v\right|^2\right\}\,.
  \end{equation*}
  On the left-hand side, we recognize the sesquilinear forms $b^\pm_R$
  of \eqref {eq:bpm-def}.  Because of periodicity in $x_2$-direction,
  we may write
  \begin{equation}    \label{eq:testing1}
    h\ \left[ b_R^-\left(v^-_R, v^-_R\right) -
      b_R^+\left(v^+_R, v^+_R\right) \right] =
    \int_\R\int_0^h \left\{ \omega^2\vartheta\left|v\right|^2 -
    a^\eps\, \vartheta\, \left|\nabla v\right|^2 \right\}\,.
  \end{equation}
  Since the right hand side is real, taking the imaginary part of
  \eqref{eq:testing1} yields
  \begin{align}
    \label{eq:testing2}
    \Im\, b_R^-\left(v^-_R, v^-_R\right) - \Im\,
    b_R^+\left(v^+_R, v^+_R\right) = 0 \,.
  \end{align}
  Relation \eqref{eq:testing2} is an energy conservation: The energy
  flux into the domain from the left must coincide with the energy
  flux out of the domain at the right.

  \smallskip {\em Step 2: Truncations and ($m\ge 1$)-waves.}  We start
  this part of the proof with an observation regarding the cut-off
  functions; we want to have them in the argument of the sesquilinear
  form.  Due to Lemma \ref{lem:Caccioppoli} and the properties of the
  cut-off functions $\eta=\eta_R$ we have
  \begin{align}
    \label{eq:strongH1convcutoff}
    \meanint_{W_R}|v^\pm_R-v^\pm_{R,\eta}|^2 + |\nabla v^\pm_R-\nabla
    v^\pm_{R,\eta}|^2\leq \frac{C}{R}\,,
  \end{align}
  and therefore, by Lemma \ref {lem:orthogonalwavenumber},
  \begin{align*}
    &b^\pm_R\left(v^\pm_R, v^\pm_R\right)
    -b_R^\pm\left(v^\pm_{R,\eta}, v^\pm_{R,\eta}\right)\\
    &\quad=b_R^\pm\left(v^\pm_R-v^\pm_{R,\eta}, v^\pm_R\right) +
    b_R^\pm\left(v^\pm_{R,\eta},v^\pm_R- v^\pm_{R,\eta}\right)
    \to 0\quad\text{as }R\rightarrow\infty\,.
  \end{align*}
  The energy conservation \eqref {eq:testing2} therefore implies that,
  as $R\rightarrow\infty$,
  \begin{align}
    \label{eq:testing3}
    \Im\, b_R^-\left(v^-_{R,\eta}, v^-_{R,\eta}\right) - \Im\,
    b_R^+\left(v^+_{R,\eta}, v^+_{R,\eta}\right) \to 0\,.
  \end{align}

  We next decompose the sesquilinear forms $b_R^\pm$ according to the
  projections of Definition \ref {def:Projections}, and suppress the
  superscript ``$\pm$'' in the projection. We exploit sesquilinearity
  of $b_R^+$ in both arguments and write
  \begin{align}
    \nonumber
    &\Im\, b_R^+\left(v^+_{R,\eta}, v^+_{R,\eta}\right)\\
    \label{eq:decomposeprojections1}
    &\quad = \Im\, b_R^+\left(\Pi^{\ev}_{m\geq 1}\left(v^+_{R,\eta}\right),
      v^+_{R,\eta}\right) +\Im\,
    b_R^+\left(\Pi^{\ev}_{m=0}\left(v^+_{R,\eta}\right),
      \Pi^{\ev}_{m\geq 1}\left(v^+_{R,\eta}\right)\right)\\
    \nonumber
    &\qquad +\Im\, b_R^+\left(\Pi^{\ev}_{m=0}\left(v^+_{R,\eta}\right),
      \Pi^{\ev}_{m=0}\left(v^+_{R,\eta}\right)\right)\,.
  \end{align}
  We want to exploit the smallness of $m\ge 1$-contributions of Lemma
  \ref{lem:only-m=0}. The regularity result of Lemma
  \ref{lem:Caccioppoli} together with the properties of the
  sesquilinear form $b^+_R$ of Lemma \ref {lem:orthogonalwavenumber}
  yield that the first term on the right hand side of \eqref
  {eq:decomposeprojections1} vanishes in the limit as $R\rightarrow
  \infty$. For the second term we apply Lemma
  \ref{lem:RegularityProjectionsev}, which provides that also the
  gradient of $\Pi^{\ev}_{m\geq 1}\left(v^+_{R,\eta}\right)$ is small;
  Lemma \ref{lem:orthogonalwavenumber} implies
  \begin{align*}
    b_R^+\left(\Pi^{\ev}_{m=0}\left(v^+_{R,\eta}\right),
      \Pi^{\ev}_{m\geq 1}\left(v^+_{R,\eta}\right)\right)\rightarrow
    0\quad\text{as }R\rightarrow\infty\,,
  \end{align*}
  i.e.\,also the second term on the right hand side of
  \eqref{eq:decomposeprojections1} vanishes in the limit.  We find
  that, as $R\rightarrow\infty$,
  \begin{align}
    \label{eq:simplification 1}
    \Im\, b_R^+\left(v^+_{R,\eta}, v^+_{R,\eta}\right) = \Im\,
    b_R^+\left(\Pi^{\ev}_{m=0}\left(v^+_{R,\eta}\right),
      \Pi^{\ev}_{m=0}\left(v^+_{R,\eta}\right)\right) +
    o(1)\,.
  \end{align}

  \smallskip {\em Step 3: Energy flux and outgoing wave conditions.}
  In this step we decompose
  $\Im\,b_R^+\left(\Pi^{\ev}_{m=0}\left(v^+_{R,\eta}\right),
    \Pi^{\ev}_{m=0}\left(v^+_{R,\eta}\right)\right)$ as follows:
  \begin{align}
    \begin{split}
      \label{eq:decomposeprojections2}
      &\Im\,b_R^+\left(\Pi^{\ev}_{m=0}\left(v^+_{R,\eta}\right),
        \Pi^{\ev}_{m=0}\left(v^+_{R,\eta}\right)\right)\\
      &\qquad =\Im\,b_R^+\left(\Pi^+_{<0}\Pi^{\ev}_{m=0}\left(v^+_{R,\eta}\right),
        \Pi^+_{<0}\Pi^{\ev}_{m=0}\left(v^+_{R,\eta}\right)\right)\\
      &\qquad\qquad +\Im\,b_R^+\left(\Pi^+_{<
          0}\Pi^{\ev}_{m=0}\left(v^+_{R,\eta}\right),
        \Pi^+_{\geq 0}\Pi^{\ev}_{m=0}\left(v^+_{R,\eta}\right)\right)\\
      &\qquad\qquad +\Im\,b_R^+\left(\Pi^+_{\geq
          0}\Pi^{\ev}_{m=0}\left(v^+_{R,\eta}\right), \Pi^+_{\geq
          0}\Pi^{\ev}_{m=0}\left(v^+_{R,\eta}\right)\right)\\
      &\qquad\qquad +\Im\,b_R^+\left(\Pi^+_{\geq
          0}\Pi^{\ev}_{m=0}\left(v^+_{R,\eta}\right),
        \Pi^+_{<0}\Pi^{\ev}_{m=0}\left(v^+_{R,\eta}\right)\right)\\
      &\qquad =\Im\,b_R^+\left(\Pi^+_{<0}\Pi^{\ev}_{m=0}\left(v^+_{R,\eta}\right),
        \Pi^+_{<0}\Pi^{\ev}_{m=0}\left(v^+_{R,\eta}\right)\right)\\
      &\qquad\qquad +\Im\,b_R^+\left(\Pi^+_{\geq
          0}\Pi^{\ev}_{m=0}\left(v^+_{R,\eta}\right), \Pi^+_{\geq
          0}\Pi^{\ev}_{m=0}\left(v^+_{R,\eta}\right)\right)\,,
    \end{split}
  \end{align}
  where the last equality holds, since for $\lambda=(j,m=0)\in
  I^+_{<0}$ and $\tilde\lambda=(\tilde j,m=0)\in I^+_{\geq 0}$ one
  always has $j\neq \tilde j$ and thus the mixed sesquilinear forms
  vanish due to orthogonality in the wave number, cf. Lemma
  \ref{lem:orthogonalwavenumber}. Exploiting the outgoing wave
  condition \eqref{eq:outgoingtruncated} on the right or, better, the
  weaker expression \eqref{eq:weakoutgoingright}, we find that the
  first term on the right hand side of \eqref
  {eq:decomposeprojections2} vanishes in the limit $R\to
  \infty$. Hence
  \begin{align}
    \begin{split}
      \label{eq:simplification2}
      &\Im\,b_R^+\left(\Pi^{\ev}_{m=0}\left(v^+_{R,\eta}\right),
        \Pi^{\ev}_{m=0}\left(v^+_{R,\eta}\right)\right)\\
      &\quad =\Im\,b_R^+\left(\Pi^+_{\geq
          0}\Pi^{\ev}_{m=0}\left(v^+_{R,\eta}\right), \Pi^+_{\geq
          0}\Pi^{\ev}_{m=0}\left(v^+_{R,\eta}\right)\right) +
      o(1)\quad\text{as }R\rightarrow \infty\,.
    \end{split}
  \end{align}
  We emphasize that we only used the energetic outgoing wave condition
  \eqref{eq:weakoutgoingright} in this calculation. 

  Combining \eqref{eq:simplification 1} with
  \eqref{eq:simplification2} we finally obtain, as
  $R\rightarrow\infty$,
  \begin{equation}\label{eq:proof-prop-bRpm}
    \Im\, b_R^+\left(v^+_{R,\eta}, v^+_{R,\eta}\right)
    =\Im\,b_R^+\left(\Pi^+_{\geq
        0}\Pi^{\ev}_{m=0}\left(v^+_{R,\eta}\right), \Pi^+_{\geq
        0}\Pi^{\ev}_{m=0}\left(v^+_{R,\eta}\right)\right) +
    o(1)\,.
  \end{equation}

  \smallskip {\em Step 4: Consequences for outgoing waves.}  We
  analyze \eqref {eq:proof-prop-bRpm} further, exploiting the discrete
  Bloch expansion of $v^\pm_{R,\eta}=\sum_{\lambda\in I_{R}}
  \alpha^\pm_{\lambda,R} U^\pm_\lambda$:
  \begin{align*}
    &\Im\, b_R^+\left(\Pi^+_{\geq
        0}\Pi^{\ev}_{m=0}\left(v^+_{R,\eta}\right),
      \Pi^+_{\geq 0}\Pi^{\ev}_{m=0}\left(v^+_{R,\eta}\right)\right)\\
    &\quad =\Im\, \sum_{\lambda=(j,0)\in I_{R}\cap I^+_{\geq
        0}}\quad \sum_{\tilde\lambda=(\tilde j,0) \in I_{R}\cap I^+_{\geq
        0}} \bar\alpha^+_{\lambda,R}\, \alpha^+_{\tilde\lambda,R}\ b_R^+
    \left(U^+_\lambda,U^+_{\tilde\lambda}\right)\\
    &\quad =\sum_{\lambda=(j,0)\in I_{R}\cap I^+_{\geq
        0}}|\alpha^+_{\lambda,R}|^2\,
    \Im\,b_R^+\left(U^+_\lambda,U^+_\lambda\right)
    =\sum_{\lambda=(j,0)\in I_{R}\cap I^+_{\geq
        0}}|\alpha^+_{\lambda,R}|^2\,P^+_\lambda\,.
  \end{align*}
  In the last line we again exploited the orthogonality of the
  sesquilinear form $b^+_R$ in the wave number $j$, see Lemma
  \ref{lem:orthogonalwavenumber}, and the relation \eqref
  {eq:relationbS-2} for $P^\pm_\lambda$. We may therefore write \eqref
  {eq:proof-prop-bRpm} as
  \begin{align*}
    \Im\, b_R^+\left(v^+_{R,\eta}, v^+_{R,\eta}\right) &=
    \sum_{\lambda=(j,0) \in I_{R}\cap I^+_{\geq
        0}}|\alpha^+_{\lambda,R}|^2\,P^+_\lambda +o(1)\,.
  \end{align*}
  On the left, we find similarly
  \begin{align*}
    \Im\, b_R^-\left(v^-_{R,\eta}, v^-_{R,\eta}\right) &=
    \sum_{\lambda=(j,0)\in I_{R}\cap I^-_{\leq
        0}}|\alpha^-_{\lambda,R}|^2\,P^-_\lambda +o(1)\,.
  \end{align*}
  The energy relation \eqref{eq:testing2} together with the sign
  properties $P^+_\lambda\ge 0$ for $\lambda \in I^+_{\geq 0}$ and
  $P^-_\lambda\le 0$ for $\lambda \in I^-_{\le 0}$ allows to conclude
  \eqref {eq:result-P-weighted-sum}.
\end{proof}

\subsection{Proof of Theorem \ref {thm:uniqueness} and Theorem \ref
  {thm:VertWaveNumber}}
\label{ssec.Bloch-uniqueness}

We study solutions to the Helmholtz equation \eqref {eq:Peps}. In
order to prove Theorems \ref {thm:uniqueness} and \ref
{thm:VertWaveNumber}, we have to check conditions that are satisfied
by the support of Bloch measures of solutions. We recall the notation
\begin{equation*}
  J^\pm_{=0,l} =\{j\in Z=[0,1]^2\, |\,P_\lambda^\pm =0\,
  \text{ for }\lambda=(j,l)\}
\end{equation*} 
for the index set corresponding to vertical waves.

\paragraph{\bf Proof of Theorem \ref {thm:uniqueness}.}
In the following, we consider solutions to the transmission Problem
\ref {prob:Transmission}. We are interested in a function $v$ which is
the difference of two solutions or, equivalently, a solution to the
problem without incoming wave.

\begin{proposition}[Solutions in absence of incoming
  waves]\label{prop:uniqueness}
  Let Assumption \ref{ass:smallnessfrequency} on $\omega>0$ be
  satisfied and let $v$ be a solution to the scattering problem
  \eqref{eq:Peps}, periodic in vertical direction, satisfying outgoing
  wave conditions on the left and on the right according to Definition
  \ref{def:outgoingwave}, without incoming wave. Let
  $\nu^\pm_{l,\infty}$, with $l\in\N_0$, be Bloch measures that are
  generated by $v$. Then
  \begin{align}
    \nu^\pm_{l,\infty} &= 0\qquad \text{for } l\ge 1,\\
    \label{eq:uniqueness1}
    \supp(\nu^\pm_{0,\infty}) &\subset J^\pm_{=0,0}\,.
  \end{align}
\end{proposition}

\begin{proof}
  We only show the statement for the limiting Bloch measures
  $\nu^+_{l,\infty}$, the argument for $\nu^-_{l,\infty}$ is
  analogous. Let $v^+_{R,\eta}=\sum_{\lambda\in
    I_R}\alpha^+_{\lambda,R}U_\lambda^+$ be the expansion of the
  truncated solution. Then the corresponding discrete Bloch measures
  are given by
  \begin{equation*}
    \nu_{l,R}^+ = \sum_{\lambda = (j,l)\in I_R}  |\alpha_{\lambda,R}^+|^2\,\delta_j\,.
  \end{equation*}

  \smallskip {\it The case $l\geq 1$:} From
  \eqref{eq:projection-m-ge-1} we know that
  \begin{equation*}
    \meanint_{W_R} \left| \Pi^{\ev}_{m\ge 1} (v^+_{R,\eta}) \right|^2
    =\sum_{\lambda=(j,m)\in I_R\atop m\ge 1}|\alpha_{\lambda,R}^+|^2
    \leq\frac{C}{R}\,,
  \end{equation*}
  and hence
  \begin{equation*}
    \int_Z\,d\nu^+_{l,R}=\sum_{\lambda = (j,l)\in
      I_R}|\alpha_{\lambda,R}^+|^2\rightarrow 0\quad\text{as
    }R\rightarrow\infty\,.
  \end{equation*}
  This shows $\nu^+_{l,\infty}=0$ for every $l\geq 1$.

  \smallskip {\it The case $l=0$:} We have to show
  $\supp{(\nu^+_{0,\infty})}\subset J^+_{=0,0}$.  To this end, we
  consider an arbitrary test function $\phi\in C(Z)$ with
  \begin{equation*}
    \supp{(\phi)} \subset\{j\in Z\,|\, 
    \lambda=(j,0)\in I^+_{<0}\cup I^+_{>0}\}\,.    
  \end{equation*}
  The outgoing wave condition \eqref {eq:outgoingright} and
  Proposition \ref{prop:weightedsumalpha} yield, in the limit $R\to
  \infty$,
  \begin{align*}
    \sum_{\lambda=(j,0)\in I_{R}\cap I^+_{< 0}}
    |\alpha^+_{\lambda,R}|^2\, \to 0\quad\text{and}\quad
    \sum_{\lambda=(j,0)\in I_{R}\cap I^+_{> 0}}
    |\alpha^+_{\lambda,R}|^2\,P^+_\lambda \to 0\,.
  \end{align*}
  In the following we assume that $\supp{(\phi)}\cap \{j\in Z\,|\,
  \lambda=(j,0)\in I^+_{>0}\}\neq \emptyset$, otherwise the proof
  simplifies. We find
  \begin{equation*}
    c_1 := \inf_{\lambda=(j,0)\in I^+_{>0} \atop
      j\in\supp{(\phi)}}{P^+_\lambda} > 0\,.
  \end{equation*}
  Without loss of generality, we assume $\phi\ge 0$ (otherwise we
  consider absolute values). For the limit $R\to \infty$ we
  calculate
  \begin{align*}
    \int_Z \phi\, d\nu^+_{0,R} &= \sum_{\lambda=(j,0)\in I^+_{<0}\cap
      I_R\atop j\in\supp{(\phi)}}|\alpha_{\lambda,R}^+|^2\phi(j)
    +\sum_{\lambda=(j,0)\in I^+_{>0}\cap I_R\atop j\in\supp{(\phi)}}
    |\alpha_{\lambda,R}^+|^2\phi(j) \\
    &\leq \|\phi\|_\infty \sum_{\lambda=(j,0)\in I_{R}\cap I^+_{<
        0}}|\alpha^+_{\lambda,R}|^2 +\|\phi\|_\infty\ \frac{1}{c_1}
    \sum_{\lambda=(j,0)\in I_{R}\cap I^+_{>
        0}}|\alpha^+_{\lambda,R}|^2\, P^+_\lambda \to 0\,.
  \end{align*}
  This shows \eqref {eq:uniqueness1} for ``+'', since $\phi$ with
  support outside $J^+_{=0,0}$ was arbitrary. The argument for ``-''
  is analogous.
\end{proof}

We next prove that, far away from the interface, solutions to the
transmission problem contain only waves that satisfy the frequency
condition and the vertical periodicity.

\begin{proposition}[Bloch measures and frequency
  condition]\label{prop:uniqueness-omega}
  Let Assumption \ref{ass:smallnessfrequency} on $\omega>0$ be
  satisfied and let $v$ be a solution to the scattering problem
  \eqref{eq:Peps}, periodic in vertical direction, satisfying outgoing
  wave conditions on the left and on the right according to Definition
  \ref{def:outgoingwave}, with incoming wave $U_\inc$ or without
  incoming wave. Let $\nu^\pm_{0,\infty}$ be the Bloch measure to
  $l=0$ that is generated by $v$. Then
  \begin{equation}
   \label{eq:Bloch-measure-1}
   \supp(\nu^\pm_{0,\infty}) \subset \left\{ j\in Z | \mu_0^\pm(j) =
     \omega^2 \,,\ j_2\in \Z/K \right\} \,.
  \end{equation}
\end{proposition}

\begin{proof} 
  Remark \ref {rem:Blochdifferentperiod} implies that the discrete
  Bloch measures $\nu^\pm_{l,R}$ are supported on $\left\{ j\in Z |\,
    j_2\in \Z/K \right\}$. This implies that also the limit measure
  $\nu^\pm_{l,\infty}$ are supported on this set.

  In order to show \eqref {eq:Bloch-measure-1}, it remains to check
  the frequency condition; we proceed as in the last proof. Let $\phi:
  Z\to \R$ be continuous and bounded with $\supp(\phi) \cap \{ j |
  \mu_0(j) = \omega^2 \} = \emptyset$. Arguing with decompositions of
  the domain of integration, we can consider separately a
  test-function $\phi\ge 0$ with the property $\phi(j) > 0 \Rightarrow
  \mu_0(j) > \omega^2$ and a test-function $\tilde\phi\ge 0$ with the
  property $\tilde\phi(j) > 0 \Rightarrow \mu_0(j) < \omega^2$. The
  arguments are analogous and we consider here only $\phi$ as above.

  By continuity of $\phi$ we find some $\delta>0$ such that $\mu_0(j)
  - \omega^2 \ge \delta$ for every $j \in \supp(\phi)$. Our aim is to
  show that $\int_Z\phi\,d\nu^+_{0,\infty}=0$.  By definition of the
  Bloch measure $\nu^+_{0,\infty}$ we have, as $R\to\infty$,
  \begin{align}
    \begin{split}
      \label{eq:proof-thm13}
      0&\leq \delta \int_Z\phi\,d\nu^+_{0,\infty}\leftarrow
      \delta \int_Z \phi\, d\nu^+_{0,R}
      = \delta \sum_{\lambda = (j,0)\in I_R} |\alpha_{\lambda,R}^+|^2 \phi(j)\\
      &\leq \sum_{\lambda = (j,0)\in I_R} (\mu_0(j)-\omega^2) 
      |\alpha_{\lambda,R}^+|^2\phi(j) \,.
    \end{split}
  \end{align}
  The result $\int_Z\phi\,d\nu^+_{0,\infty}=0$ is shown once we prove
  that the right hand side of \eqref{eq:proof-thm13} vanishes in the
  limit $R\to \infty$. In order to show this fact, we recall that the
  coefficients $\alpha_{\lambda,R}^+$ are obtained from a
  Bloch-expansion of the solution at the far right,
  i.e. $\alpha^+_{\lambda,R}=\langle u^+_{R,\eta},
  U^+_\lambda\rangle_R$. We calculate
  \begin{align*}
    &\sum_{\lambda = (j,0)\in I_R} (\mu_0(j)-\omega^2) |\alpha_{\lambda,R}^+|^2\phi(j) \\
    &\qquad\stackrel{(1)}{=}\sum_{\lambda = (j,0)\in I_R}\phi(j)\,
    \overline{\alpha_{\lambda,R}^+}\,\left[\langle u^+_{R,\eta},
      \mu_0(j) U^+_\lambda\rangle_R
      -\langle \omega^2\,u^+_{R,\eta}, U^+_\lambda\rangle_R\right]\displaybreak[2]\\
    &\qquad\stackrel{(2)}{=}\sum_{\lambda = (j,0)\in I_R}\phi(j)\,
    \overline{\alpha_{\lambda,R}^+}\,\left[\langle u^+_{R,\eta},
      \calL_0 U^+_\lambda\rangle_R
      -\langle \omega^2\,u^+_{R,\eta}, U^+_\lambda\rangle_R\right]\displaybreak[2]\\
    &\qquad\stackrel{(3)}{=}\sum_{\lambda =(j,0)\in I_R}\phi(j)\,
    \overline{\alpha_{\lambda,R}^+}\,\langle\calL_0 u^+_{R,\eta}
    -\omega^2\,u^+_{R,\eta},  U^+_\lambda\rangle_R\displaybreak[2]\\
    &\qquad\stackrel{(4)}{\leq}\|\phi\|_{\infty}\,\left(\sum_{\lambda
        =(j,0)\in I_R}|\alpha_{\lambda,R}^+|^2\right)^{1/2}
    \left(\sum_{\lambda =(j,0)\in I_R}\left|\langle\calL_0
        u^+_{R,\eta} -\omega^2\,u^+_{R,\eta},
        U^+_\lambda\rangle_R\right|^2\right)^{1/2}.
  \end{align*}
  In this calculation we used the following: (1) formula for
  $\alpha^+_{\lambda,R}$, (2) the eigenvalue property of $U_\lambda$
  with eigenvalue $\mu_\lambda = \mu_m(j)$, (3) integration by parts
  without boundary terms due to the cut-off function $\eta$, (4)
  Cauchy-Schwarz inequality.  Using orthonormality of the basis
  functions $U^\pm_\lambda$ we obtain
  \begin{align*}
    &\sum_{\lambda = (j,0)\in I_R} (\mu_0(j)-\omega^2) |\alpha_{\lambda,R}^+|^2\phi(j)\\
    &\qquad \leq \|\phi\|_\infty\left(\meanint_{W_R}
      \left|\Pi^{\ev}_{m=0}u^+_{R,\eta}\right|^2\right)^{1/2}
    \left(\meanint_{W_R}\left|\Pi^{\ev}_{m=0}\left(\calL_0
          u^+_{R,\eta}
          -\omega^2\,u^+_{R,\eta}\right)\right|^2\right)^{1/2}\displaybreak[2]\\
    &\qquad \leq
    \|\phi\|_\infty\left(\meanint_{W_R}\left|u^+_{R,\eta}\right|^2\right)^{1/2}
    \left(\meanint_{W_R}\left|\calL_0 u^+_{R,\eta}
        -\omega^2\,u^+_{R,\eta}\right|^2\right)^{1/2}.
  \end{align*}
  Since $u^+_{R,\eta}$ satisfies uniform $L^2$-bounds and since
  $\calL_0 u^+_{R,\eta} = \omega^2\,u^+_{R,\eta}$ holds up to a small
  $L^2$-error, the right hand side of \eqref {eq:proof-thm13} is small
  for large $R>0$. This proves $\int_Z\phi\,d\nu^+_{0,\infty} = 0$ and
  hence \eqref {eq:Bloch-measure-1} for ``+''.  The proof for ``-'' is
  analogous.
\end{proof}

\begin{proof}[Proof of Theorem \ref {thm:uniqueness}]
  The difference $v$ of two solutions satisfies the outgoing wave
  condition without an incident wave.  Theorem \ref {thm:uniqueness}
  is an immediate consequence of Propositions \ref {prop:uniqueness}
  and \ref {prop:uniqueness-omega}.
\end{proof}

\paragraph{\bf Proof of Theorem \ref {thm:VertWaveNumber}.}

We now provide the proof of Theorem \ref {thm:VertWaveNumber}. We
therefore assume that: Assumption \ref {ass:smallnessfrequency} on
$\omega>0$ is satisfied and $u$ is a solution of the scattering
problem with incoming wave $U_\inc$, which has the wave number $k =
(k_1, k_2)$. In particular, $u$ is a vertically periodic solution of
\eqref {eq:Peps} such that $u$ and $u-U_\inc$ satisfy the outgoing
wave conditions on the right and on the left.

Let $\nu^\pm_{l,\infty}$ be Bloch measures that are generated by the
solution $u$.  The frequency condition \eqref{eq:smallnessfrequency}
is satisfied and we can therefore use Lemma \ref {lem:only-m=0}.  As
in Proposition \ref {prop:uniqueness}, case $l\ge 1$, we conclude from
\eqref {eq:projection-m-ge-1} (and the analogous result for ``-'')
that $\nu^\pm_{l,\infty} = 0$ holds for every $l\ge 1$. Moreover,
according to Proposition \ref{prop:uniqueness-omega} we have that
$\supp(\nu^\pm_{0,\infty}) \subset \left\{ j\in Z | \mu_0^\pm(j)
  =\omega^2\,,\, j_2\in \Z/K \right\}$.

Theorem \ref {thm:VertWaveNumber} is shown once we verify the
following property of the Bloch measure $\nu^\pm_{0,\infty}$:
\begin{equation}
   \label{eq:Bloch-measure-vert-1}
   \supp(\nu^\pm_{0,\infty}) \subset \left\{ j\in Z | j_2 = k_2
   \right\}  \cup J^\pm_{=0,0}\,.
\end{equation}

\smallskip {\em Proof of \eqref {eq:Bloch-measure-vert-1}.} We
consider the projection $\Pi^{\ver}_{k_2} u$ of $u$. This function is
again a solution of the scattering problem.  Indeed, by Lemma
\ref{lem:Propertiesvertprojection} one has $\Pi^{\ver}_{k_2} u\in
H^1_\loc(\R\times(0,h);\C)$ with periodicity in the $x_2$-variable,
and for arbitrary test functions $\varphi\in C_c^\infty
(\R\times(0,h))$ there holds
\begin{align*}
  &\int_\R\int_0^h\nabla\varphi\cdot
  a^\eps\,\nabla\left(\Pi^{\ver}_{k_2} u\right)
  =\int_\R\int_0^h\nabla\varphi\cdot a^\eps\,\Pi^{\ver}_{k_2}
  \left(\nabla u\right)
  =\int_\R\int_0^h\Pi^{\ver}_{k_2}
  \left(\nabla\varphi\right)\cdot a^\eps\, \nabla u\\
  &\quad
  =\int_\R\int_0^h\nabla\left(\Pi^{\ver}_{k_2}\varphi\right)\cdot
  a^\eps\, \nabla u =\omega^2\int_\R\int_0^h
  \Pi^{\ver}_{k_2}\varphi\,u
  =\omega^2\int_\R\int_0^h\varphi\,\Pi^{\ver}_{k_2}u\,,
\end{align*}
where we exploited the orthogonality properties of $\Pi^{\ver}_{k_2}$
from Lemma \ref{lem:orthweight} and the solution property of $u$.

As a consequence, the difference $v := u - \Pi^{\ver}_{k_2} u$ is a
solution of the scattering problem with vanishing incoming wave (just
as the difference of two solutions in the proof of Theorem \ref
{thm:uniqueness}). The uniqueness statement of Proposition \ref
{prop:uniqueness} implies: Bloch measures (for $l=0$) that are
generated by $v$ have their support in vertical waves, i.e.\,in
$J^\pm_{=0,0}$.

On the other hand, the Bloch measure of $\Pi^{\ver}_{k_2} u$ is
concentrated on waves with vertical wave number $k_2$, i.e.\,in
$\left\{ j\in Z | j_2 = k_2 \right\}$. This follows immediately from
the fact that all coefficients $\alpha_{(j,m)}$ with $j_2\neq k_2$ in
the expansion of $\Pi^{\ver}_{k_2} u$ vanish.

Since the Bloch measure of $u$ can have its support only in the union
of the supports corresponding to $\Pi^{\ver}_{k_2} u$ and $u -
\Pi^{\ver}_{k_2} u$, the claim \eqref {eq:Bloch-measure-vert-1}
follows.

Theorem \ref {thm:VertWaveNumber} is shown.

\section{Outlook and conclusions}

\subsubsection*{Remarks on the existence of solutions}
\label{sec.existence}

We give some remarks concerning the existence of solutions to the
scattering problem. In the end, our radiation condition is ``the right
one'' only if, besides uniqueness, an existence result can be shown.

We formulate the following conjecture: Given a non-singular frequency
$\omega>0$, given coefficients $a = a^\eps$ that are equal to $1$ in
the left half plane and $\eps$-periodic in the right half plane,
strictly positive and bounded, given finally an incoming wave $U_\inc$
as in \eqref {eq:u-inc} (possibly with a condition on $k$), there
exists a solution $u$ of the transmission Problem \ref
{prob:Transmission}.

\smallskip The idea for an existence proof is the limiting absorption
principle (see e.g.\,\cite {Hoang2011, JolyLiFliss2006, Radosz2015}
for recent contributions): For a positive artificial damping parameter
$\delta>0$, we consider the equation
\begin{equation}
  \label{eq:damped-HH}
  -\nabla\cdot ((1-i\delta)a(x) \nabla u^\delta(x)) = \omega^2 u^\delta(x)
\end{equation}
for $x\in \Omega = \R\times (0,h)$.  Due to the strictly negative
imaginary part of the coefficient $(1-i\delta)a(x)$, this equation
admits a unique solution $u^\delta$ in the Beppo-Levi space $\dot
H^1(\Omega)$ as can be shown with the Lax-Milgram Lemma.

To proceed, two properties must be shown. The first is: The sequence
$u^\delta$ satisfies estimates in some function space, uniformly in
$\delta>0$. Once this is shown, we can consider the distributional
limit $u$ of the sequence $u^\delta$ as $\delta\to 0$. As a
consequence of distributional convergence, the limit $u$ is a solution
of the Helmholtz equation with coefficients $a$.

The intricate part of this approach is to show the second property:
The limit $u$ satisfies the outgoing wave condition. We do not see a
straightforward argument that yields this condition.

\subsubsection*{Our outgoing wave condition in a numerical scheme}

The condition of Defitition \ref {def:outgoingwave} might seem
inadequate for numerical purposes on first sight. But it is possible
to interpret the condition for a discrete realization: Instead of
demanding that limits $R\to\infty$ vanish, we rather demand that, for
some large distance $R$, the left hand side of \eqref
{eq:outgoingright} vanishes. In the numerical scheme, this amounts to
imposing that a finite number of projections to left-going waves
vanishes.

\subsubsection*{Conclusions}

We have investigated the transmission properties at the boundary of a
photonic crystal. Our theorems justify the following: An incoming wave
generates, inside the photonic crystal, only those Bloch waves, for
which the eigenvalue coincides with the (squared) frequency of the
incoming wave. Furthermore, only those Bloch waves can be generated
that have the same vertical wave number as the incoming wave; this
latter statement is true up to vertical waves.

Our results rely on a new outgoing wave condition in photonic
crystals. The new radiation condition is based on Bloch expansions. It
is accompanied by a (weak) uniqueness result, which is expressed with
Bloch-measures. The uniqueness result is the basis for the analysis of
the transmission problem.

\subsection*{Acknowledgements}

Support of both authors by DFG grant Schw 639/6-1 is gratefully
acknowledged.

\appendix
\section{Orthogonality and regularity properties}

\begin{lemma}[Orthogonality with periodic weight]
  \label{lem:orthweight}
  Let $f:\R\rightarrow\C$ be $\eps$-periodic and integrable, let
  $R\in\N$ be an integer.
  \begin{enumerate}
  \item {\em Orthogonality of exponentials.}  Let $j,\tilde j\in Q_R$
    with $j\neq \tilde j$. Then
    \begin{align}
      \label{eq:orthweight}
      \int_{0}^{\eps R} f(y)e^{2\pi i j y/\eps}e^{-2\pi i \tilde j y/\eps}\,dy=0\,.
    \end{align}
  \item {\em Orthogonality of the vertical pre-Bloch projection.}  Let
    $u,v\in L^2_\loc(\R\times(0,\eps R);\C)$ and let $k_2\in
    Q_R$. Then there holds
    \begin{align}
      \label{eq:orthogonalweightproject}
      \int_0^{\eps R} f(y)\,u(x_1,y)\,\overline{\Pi^{\ver}_{k_2} v(x_1,y)}\,dy
      &=\int_0^{\eps R} f(y)\Pi^{\ver}_{k_2}
      u(x_1,y)\,\overline{\Pi^{\ver}_{k_2} v(x_1,y)}\,dy\,.
    \end{align}
  \end{enumerate}
\end{lemma}

\begin{proof}
  1. By dividing the interval $(0,\eps R)$ into subintervals of length
  $\eps$, we obtain
  \begin{align*}
    &\int_{0}^{\eps R} f(y)e^{2\pi i j y/\eps}e^{-2\pi i \tilde j
      y/\eps}\,dy
    =\sum_{k=0}^{R-1}\int_{k\eps}^{(k+1)\eps}f(y)e^{2\pi i(j-\tilde j)y/\eps}\, dy\\
    &\qquad =\sum_{k=0}^{R-1}\int_0^\eps f(y+k\eps)e^{2\pi i(j-\tilde
      j)(y+k\eps)/\eps}\,dy =\sum_{k=0}^{R-1}e^{2\pi i(j-\tilde
      j)k}\int_0^\eps f(y)e^{2\pi i(j-\tilde j)y/\eps}\,dy\,,
  \end{align*}
  where in the last equality we exploited the periodicity of the
  weight $f$. By setting $C(j,\tilde j):=\int_0^\eps f(y)e^{2\pi
    i(j-\tilde j)y/\eps}\,dy$ we conclude
  \begin{align*}
    \int_{0}^{\eps R} f(y)e^{2\pi i j y/\eps}e^{-2\pi i \tilde j
      y/\eps}\,dy =C(j,\tilde j)\sum_{k=0}^{R-1}\left(e^{2\pi
        i(j-\tilde j)}\right)^k =C(j,\tilde j)\frac{1-e^{2\pi
        i(j-\tilde j)R}}{1-e^{2\pi i(j-\tilde j)}} =0\,.
  \end{align*}
  In the last step we used $j,\tilde j\in Q_R$, which implies $R
  (j-\tilde j) \in\Z$ and $j,\tilde j < 1$, and exploited $j\neq
  \tilde j$.
	
  \smallskip 2.  Let $u, v$ have vertical pre-Bloch expansions
  \begin{align*}
    u(x_1,x_2) =\sum_{j_2\in Q_{R}}\Phi_{j_2}(x_1, x_2)\, e^{2\pi i
      j_2 x_2/\eps}\,,
    \quad v(x_1,x_2) =\sum_{\tilde j_2\in Q_{R}}
    \tilde\Phi_{\tilde j_2}(x_1, x_2)\, e^{2\pi i \tilde j_2 x_2/\eps}\,.
  \end{align*}
  Then the left hand side of \eqref {eq:orthogonalweightproject} reads
  \begin{align*}
    &\int_0^{\eps R} f(y)u(x_1,y)\,\overline{\Pi^{\ver}_{k_2} v(x_1,y)}\,dy\\
    &\quad= \sum_{j_2\in Q_{R}}\int_0^{\eps R}
    f(y)\Phi_{j_2}(x_1,y)\, e^{2\pi i j_2 y/\eps}\,
    \overline{\tilde\Phi_{k_2}(x_1, y)}\, e^{-2\pi i k_2 y/\eps}\,dy\,.
  \end{align*}
  Since the function $f(\cdot)\Phi_{j_2}(x_1,
  \cdot)\overline{\tilde\Phi_{k_2}(x_1, \cdot)}$ is $\eps$-periodic,
  we can apply the orthogonality \eqref {eq:orthweight} of Item 1. The
  sum on the right hand side collapses to $j_2 = k_2$ and we find
  \eqref{eq:orthogonalweightproject}.
\end{proof}

\begin{lemma}[Vertical pre-Bloch projection and gradients]
  \label{lem:Propertiesvertprojection}
  Let $K\in \N$, $h=\eps K$, and $k_2\in Q_{K}$. Let $u\in
  H_\loc^1(\R\times (0,h);\C)$ be periodic in the $x_2$-variable. Then
  the function $\Pi^{\ver}_{k_2}u\in H^1_\loc(\R\times (0,h);\C)$ is
  periodic in $x_2$ and there holds
  \begin{align}
    \label{eq:commuteprojectiongradient}
    \nabla\left(\Pi^{\ver}_{k_2}u\right)
    =\Pi^{\ver}_{k_2}\left(\nabla u\right)\,.
  \end{align}
\end{lemma}

\begin{proof}
  Let $u$ have the pre-Bloch expansion $u(x_1,x_2)=\sum_{j_2\in
    Q_{K}}\Phi_{j_2}(x_1, x_2)\, e^{2\pi i j_2 x_2/\eps}$.  Due to the
  periodicity of $u$ in the $x_2$-variable, each $\Phi_{j_2}$ in the
  above (finite) sum has $H^1$-regularity, and thus
  \begin{align}
    \begin{split}
      \label{eq:expansiongradu}
      \nabla u(x_1,x_2)&=\sum_{j_2\in Q_{K}}\nabla\left(\Phi_{j_2}(x_1, x_2)\, 
        e^{2\pi i j_2 x_2/\eps}\right)\\
      &=\sum_{j_2\in Q_{K}}\left[\nabla\Phi_{j_2}(x_1, x_2)+2\pi i
        j_2/\eps\,\Phi_{j_2}(x_1, x_2)e_2\right] e^{2\pi i j_2
        x_2/\eps}\,,
    \end{split}
  \end{align}
  where $e_2 = (0,1) \in\R^2$ denotes the second unit vector. Since
  the expression in the squared brackets is $\eps$-periodic,
  \eqref{eq:expansiongradu} is an expansion of $\nabla u$; uniqueness
  of the pre-Bloch expansion implies
  \begin{align*}
    \Pi^{\ver}_{k_2}\left(\nabla u\right)(x_1,x_2)&=
    \left(\nabla\Phi_{k_2}(x_1, x_2)+2\pi i k_2/\eps\,\Phi_{k_2}(x_1,
      x_2)e_2 \right)
    e^{2\pi i k_2 x_2/\eps}\\
    &=\nabla\left(\Phi_{k_2}(x_1, x_2)\, e^{2\pi i k_2
        x_2/\eps}\right)
    =\nabla\left(\Pi^{\ver}_{k_2}u\right)(x_1,x_2)\,,
  \end{align*}
  which proves \eqref{eq:commuteprojectiongradient}.
\end{proof}

\begin{lemma}[Caccioppoli estimate] 
  \label{lem:Caccioppoli}
  Let $u\in L^2_\loc(\R\times (0,h))$ be a vertically periodic
  solution of the Helmholtz equation $\mathcal{L}_0u=\omega^2 u$. Let
  $u$ satisfy the uniform $L^2$-bounds of Definition
  \ref{def:outgoingwave}. Then there holds
  \begin{align}
    \label{eq:caccioppoli}
    \frac{1}{R}\int_{W_R\setminus W_{R-1}}|u_R^\pm|^2 + |\nabla
    u_R^\pm|^2\leq C\quad\text{and}\quad \meanint_{W_R}|u_R^\pm|^2 +
    |\nabla u_R^\pm|^2\leq C
  \end{align}
  with $C$ independent of $R$.
\end{lemma}

\begin{proof}
  The proof is, up to translations and a summation, analogous to the
  proof of the standard Caccioppoli estimate: On a rectangle $(L-1,
  L+2)\times (0,h)$ we use a cut-off function $\theta$ with compact
  support that depends only on $x_1$ and which is identical $1$ on
  $(L, L+1)\times (0,h)$. Testing the equation with $\theta^2 \bar u$
  provides
  \begin{align*}
    \int_{L-1}^{L+2}\int_0^h \omega^2 |u|^2 \theta^2 
    = \int_{L-1}^{L+2}\int_0^h \calL_0 u (\theta^2 \bar u)
    = \int_{L-1}^{L+2}\int_0^h \left\{ a^\eps |\nabla u|^2 \theta^2 
    +2 a^\eps (\nabla u \theta)\cdot (\nabla \theta \bar u)\right\}\,.
  \end{align*}
  The Cauchy-Schwarz inquality is used to treat the last term, the
  first factor is absorbed with Young's inquality in the gradient
  term, the other consists (up to bounded factors) only of the
  $L^2$-norm of $u$. We conclude that a bound for the $L^2$-norm on
  $(L-1, L+2)\times (0,h)$ implies a bound for the $L^2$-norm of the
  gradient on $(L, L+1)\times (0,h)$. A summation over many squares
  yields the result.
\end{proof}

\begin{lemma}[Regularity of eigenvalue projections $\Pi^{{\ev}}$]
  \label{lem:RegularityProjectionsev}
  Let $(v_R)_{\R\in\N}$ be a sequence of functions with
  $H^2$-regularity and vanishing boundary data, i.e.  $v_R\in
  H_0^2(W_R;\C)$. We assume that
  \begin{align}
    \label{eq:estimatesuassumption}
    \meanint_{W_R}\left|v_R\right|^2 +\left|\nabla v_R\right|^2+
    \left|\mathcal{L}_0(v_R)\right|^2\leq C_0
  \end{align}
  holds for $\mathcal{L}_0=-\nabla\cdot(a^\eps\,\nabla)$ with some
  $R$-independent constant $C_0$. 
  \begin{enumerate}
  \item Let $\Pi$ be any of the projections of Definition
    \ref{def:Projections}.  Then there exists an $R$-independent
    constant $C$ such that
    \begin{align}
      \label{eq:regularityestimate1}
      \meanint_{W_R}\left|\nabla\left(\Pi^{\ev,\pm}_{m=0}
          v_R\right)\right|^2 +
      \left|\nabla\left(\Pi^{\ev,\pm}_{m\geq 1}
          v_R\right)\right|^2
      +\left|\nabla\left(\Pi\left(\Pi^{\ev,\pm}_{m=0}
            v_R\right)\right)\right|^2 \leq C\,.
    \end{align}
  \item If, additionally, $\textmean_{W_R}\left|\Pi^{\ev,\pm}_{m\geq
        1} v_R\right|^2\rightarrow 0$ as $R\rightarrow\infty$, then
    there holds
    \begin{align}
      \label{eq:regularityestimate2}
      \meanint_{W_R}\left|\nabla\left(\Pi^{\ev,\pm}_{m\geq 1}
          v_R\right)\right|^2\rightarrow 0\quad\text{ as
      }R\rightarrow\infty\,.
    \end{align}
  \end{enumerate}
\end{lemma}

\begin{proof}
  1.  We omit the superscripts $\pm$. Concerning
  \eqref{eq:regularityestimate1} we note that, because of
  $\Pi^{{\ev}}_{m\geq 1} v_R=v_R-\Pi^{{\ev}}_{m=0} v_R$, the estimate
  for $\Pi^{{\ev}}_{m\geq 1} v_R$ follows directly from the estimate
  for $\Pi^{{\ev}}_{m=0} v_R$ and Assumption
  \eqref{eq:estimatesuassumption}.

  Since $\Pi^{{\ev}}_{m=0} v_R=\sum_{\lambda=(j,0)\in
    I_R}\alpha_\lambda U_\lambda$ is a finite sum of periodic
  functions, we find that $\Pi^{{\ev}}_{m=0} v_R$ is periodic in
  $W_R$.  This allows to calculate, with $0<a_*\leq \inf a^\eps$,
  \begin{align*}
    &a_* \meanint_{W_R}\left|\nabla\left(\Pi^{{\ev}}_{m=0}
        v_R\right)\right|^2
    \leq \meanint_{W_R}a^\eps \nabla\left(\Pi^{{\ev}}_{m=0} v_R\right)\cdot \overline{\nabla\left(\Pi^{{\ev}}_{m=0} v_R\right)}\\
    &\quad\stackrel{(1)}{=}
    \meanint_{W_R}\mathcal{L}_0\left(\Pi^{{\ev}}_{m=0}
      v_R\right)\,\overline{\Pi^{{\ev}}_{m=0} v_R}
    \stackrel{(2)}{=} \meanint_{W_R}\Pi^{{\ev}}_{m=0}\left(\mathcal{L}_0 v_R\right)\,\overline{\Pi^{{\ev}}_{m=0} v_R}\\
    &\quad\leq
    \left(\meanint_{W_R}\left|\Pi^{{\ev}}_{m=0}\left(\mathcal{L}_0
          v_R\right)\right|^2\right)^{1/2}
    \left(\meanint_{W_R}\left|\Pi^{{\ev}}_{m=0} v_R\right|^2\right)^{1/2}\\
    &\quad\leq \left(\meanint_{W_R}\left|\mathcal{L}_0
        v_R\right|^2\right)^{1/2}
    \left(\meanint_{W_R}\left|v_R\right|^2\right)^{1/2}\leq C_0\,.
  \end{align*}
  In (1) we exploited the periodicity of $\Pi^{{\ev}}_{m=0} v_R$ to
  perform integration by parts without boundary terms.  In (2), we
  used the periodicity of $v_R$, which yields
  $\mathcal{L}_0\left(\Pi^{{\ev}}_{m=0} v_R\right) =
  \Pi^{{\ev}}_{m=0}\left(\mathcal{L}_0 v_R\right)$, as shown in
  \eqref{eq:L0-applied-to-w}.  In the last line we exploited the
  norm-boundedness of projections.  The claim for
  $\Pi\left(\Pi^{\ev}_{m=0} v_R\right)$ is shown analogously, using
  again periodicity.  This concludes the proof of Relation
  \eqref{eq:regularityestimate1}.

  \smallskip 2.  The proof of Relation \eqref{eq:regularityestimate2}
  is similar and can be interpreted as an interpolation between
  function spaces. Once more, we exploit that $v_R$ has vanishing (and
  thus periodic) boundary data and that $\Pi^{{\ev}}_{m=0} v_R$ is
  periodic as a finite sum (see Item 1.). Therefore also the
  difference $\Pi^{{\ev}}_{m\geq 1} v_R=v_R-\Pi^{{\ev}}_{m=0} v_R$ is
  periodic. Arguing as above we obtain, as $R\rightarrow \infty$,
  \begin{align}
    \label{eq:estgradientmgeq1}
    a_*\, \meanint_{W_R}\left|\nabla\left(\Pi^{{\ev}}_{m\geq 1}
        v_R\right)\right|^2 \leq \left(\meanint_{W_R}\left|\mathcal{L}_0
        v_R\right|^2\right)^{1/2}
    \left(\meanint_{W_R}\left|\Pi^{{\ev}}_{m\geq
          1}v_R\right|^2\right)^{1/2}\rightarrow 0\,.
  \end{align} 
  This shows \eqref {eq:regularityestimate2} and concludes the proof.
\end{proof}

\bibliographystyle{abbrv} 
\bibliography{lit-Helmholtz-3}

\def\cprime{$'$}
\begin{thebibliography}{10}

\bibitem{Allaire-Conca-1998}
G.~Allaire and C.~Conca.
\newblock Bloch wave homogenization and spectral asymptotic analysis.
\newblock {\em J. Math. Pures Appl. (9)}, 77(2):153--208, 1998.

\bibitem{MR2533955}
G.~Allaire, M.~Palombaro, and J.~Rauch.
\newblock Diffractive behavior of the wave equation in periodic media: weak
  convergence analysis.
\newblock {\em Ann. Mat. Pura Appl. (4)}, 188(4):561--589, 2009.

\bibitem{MR2847530}
G.~Allaire, M.~Palombaro, and J.~Rauch.
\newblock Diffractive geometric optics for {B}loch wave packets.
\newblock {\em Arch. Ration. Mech. Anal.}, 202(2):373--426, 2011.

\bibitem{Ammari-Antenna-2001}
H.~Ammari, N.~B{\'e}reux, and E.~Bonnetier.
\newblock Analysis of the radiation properties of a planar antenna on a
  photonic crystal substrate.
\newblock {\em Math. Methods Appl. Sci.}, 24(13):1021--1042, 2001.

\bibitem{AmmariSantosa2004}
H.~Ammari and F.~Santosa.
\newblock Guided waves in a photonic bandgap structure with a line defect.
\newblock {\em SIAM J. Appl. Math.}, 64(6):2018--2033, 2004.

\bibitem{Bonnet-Ben-etal-SIAP2009}
A.-S. Bonnet-Ben~Dhia, G.~Dakhia, C.~Hazard, and L.~Chorfi.
\newblock Diffraction by a defect in an open waveguide: a mathematical analysis
  based on a modal radiation condition.
\newblock {\em SIAM J. Appl. Math.}, 70(3):677--693, 2009.

\bibitem{BFe2}
G.~Bouchitt{\'e} and D.~Felbacq.
\newblock Negative refraction in periodic and random photonic crystals.
\newblock {\em New J. Phys}, 7(159, 10.1088), 2005.

\bibitem{BouchitteSchweizer-Max}
G.~Bouchitt{\'e} and B.~Schweizer.
\newblock Homogenization of {M}axwell's equations in a split ring geometry.
\newblock {\em Multiscale Model. Simul.}, 8(3):717--750, 2010.

\bibitem{CastroZuazua1996}
C.~Castro and E.~Zuazua.
\newblock Une remarque sur l'analyse asymptotique spectrale en
  homog\'en\'eisation.
\newblock {\em C. R. Acad. Sci. Paris S\'er. I Math.}, 322(11):1043--1047,
  1996.

\bibitem{ChenLipton2010}
Y.~Chen and R.~Lipton.
\newblock Tunable double negative band structure from non-magnetic coated rods.
\newblock {\em New Journal of Physics}, 12(8):083010, 2010.

\bibitem{ColtonKress}
D.~Colton and R.~Kress.
\newblock {\em Inverse acoustic and electromagnetic scattering theory},
  volume~93 of {\em Applied Mathematical Sciences}.
\newblock Springer-Verlag, Berlin, second edition, 1998.

\bibitem{EfrosPokrovsky-SolidState-2004}
A.~Efros and A.~Pokrovsky.
\newblock Dielectric photonic crystal as medium with negative electric
  permittivity and magnetic permeability.
\newblock {\em Solid State Communications}, 129:643--647, 2004.

\bibitem{FigotinKlein1998}
A.~Figotin and A.~Klein.
\newblock Midgap defect modes in dielectric and acoustic media.
\newblock {\em SIAM J. Appl. Math.}, 58(6):1748--1773 (electronic), 1998.

\bibitem{Fliss-2013}
S.~Fliss.
\newblock A {D}irichlet-to-{N}eumann approach for the exact computation of
  guided modes in photonic crystal waveguides.
\newblock {\em SIAM J. Sci. Comput.}, 35(2):B438--B461, 2013.

\bibitem{Joly-Fliss-2012}
S.~Fliss and P.~Joly.
\newblock Wave propagation in locally perturbed periodic media (case with
  absorption): numerical aspects.
\newblock {\em J. Comput. Phys.}, 231(4):1244--1271, 2012.

\bibitem{FlissJoly2016}
S.~Fliss and P.~Joly.
\newblock Solutions of the time-harmonic wave equation in periodic waveguides:
  asymptotic behaviour and radiation condition.
\newblock {\em Arch. Ration. Mech. Anal.}, 219(1):349--386, 2016.

\bibitem{FlissKlindworthSchmidt2015}
S.~Fliss, D.~Klindworth, and K.~Schmidt.
\newblock Robin-to-{R}obin transparent boundary conditions for the computation
  of guided modes in photonic crystal wave-guides.
\newblock {\em BIT}, 55(1):81--115, 2015.

\bibitem{Helmholtz-1860}
H.~Helmholtz.
\newblock Theorie der {L}uftschwingungen in {R}\"ohren mit offenen {E}nden.
\newblock {\em J. Reine Angew. Math.}, 57:1--72, 1860.

\bibitem{Hoang2011}
V.~Hoang.
\newblock The limiting absorption principle for a periodic semi-infinite
  waveguide.
\newblock {\em SIAM J. Appl. Math.}, 71(3):791--810, 2011.

\bibitem{HoangRadosz-2014}
V.~Hoang and M.~Radosz.
\newblock Absence of bound states for waveguides in two-dimensional periodic
  structures.
\newblock {\em J. Math. Phys.}, 55(3):033506, 20, 2014.

\bibitem{Jaeger-1967}
W.~J{\"a}ger.
\newblock Zur {T}heorie der {S}chwingungsgleichung mit variablen
  {K}oeffizienten in {A}ussengebieten.
\newblock {\em Math. Z.}, 102:62--88, 1967.

\bibitem{PhotonicCrystals-book}
J.~Joannopoulos, S.~Johnson, J.~Winn, and R.~Meade.
\newblock {\em Photonic Crystals -- Molding the Flow of Light}.
\newblock Princeton University Press, 2008.

\bibitem{Joly-PML}
P.~Joly.
\newblock An elementary introduction to the construction and the analysis of
  perfectly matched layers for time domain wave propagation.
\newblock {\em S$\vec{\rm e}$MA J.}, 57:5--48, 2012.

\bibitem{JolyLiFliss2006}
P.~Joly, J.-R. Li, and S.~Fliss.
\newblock Exact boundary conditions for periodic waveguides containing a local
  perturbation.
\newblock {\em Commun. Comput. Phys.}, 1, 2006.

\bibitem{Lamacz-Schweizer-Max}
A.~Lamacz and B.~Schweizer.
\newblock Effective {M}axwell equations in a geometry with flat rings of
  arbitrary shape.
\newblock {\em SIAM J. Math. Anal.}, 45(3):1460--1494, 2013.

\bibitem{Lamacz-Schweizer-Neg}
A.~Lamacz and B.~Schweizer.
\newblock A negative index meta-material for {M}axwell´s equations.
\newblock {\em SIAM J. Math. Anal.}, to appear, 2017.

\bibitem{PhysRevB.65.201104}
C.~Luo, S.~G. Johnson, J.~D. Joannopoulos, and J.~B. Pendry.
\newblock All-angle negative refraction without negative effective index.
\newblock {\em Phys. Rev. B}, 65:201104, {M}ay 2002.

\bibitem{PhysRevB.44.10961}
R.~D. Meade, K.~D. Brommer, A.~M. Rappe, and J.~D. Joannopoulos.
\newblock Electromagnetic bloch waves at the surface of a photonic crystal.
\newblock {\em Phys. Rev. B}, 44:10961--10964, Nov 1991.

\bibitem{PhysRevB.69.121402}
E.~Moreno, F.~J. Garc\'{i}a-Vidal, and L.~Mart\'{i}n-Moreno.
\newblock Enhanced transmission and beaming of light via photonic crystal
  surface modes.
\newblock {\em Phys. Rev. B}, 69:121402, Mar 2004.

\bibitem{Nazarov2014}
S.~A. Nazarov.
\newblock Umov-{M}andel\cprime shtam radiation conditions in elastic periodic
  waveguides.
\newblock {\em Mat. Sb.}, 205(7):43--72, 2014.

\bibitem{Pendry2000}
J.~Pendry.
\newblock Negative refraction makes a perfect lens.
\newblock {\em Phys. Rev. Lett.}, 85(3966), 2000.

\bibitem{Pokrovsky2003333}
A.~Pokrovsky and A.~Efros.
\newblock Diffraction theory and focusing of light by a slab of left-handed
  material.
\newblock {\em Physica B: Condensed Matter}, 338(1--4):333--337, 2003.

\bibitem{Radosz2015}
M.~Radosz.
\newblock New limiting absorption and limit amplitude principles for periodic
  operators.
\newblock {\em Z. Angew. Math. Phys.}, 66(2):253--275, 2015.

\bibitem{Rellich-1943}
F.~Rellich.
\newblock \"{U}ber das asymptotische {V}erhalten der {L}\"osungen von {$\Delta
  u+\lambda u=0$} in unendlichen {G}ebieten.
\newblock {\em Jber. Deutsch. Math. Verein.}, 53:57--65, 1943.

\bibitem{Schot-Sommerfeld-1992}
S.~H. Schot.
\newblock Eighty years of {S}ommerfeld's radiation condition.
\newblock {\em Historia Math.}, 19(4):385--401, 1992.

\bibitem{Schweizer-Helmholtz-resonator}
B.~Schweizer.
\newblock The low-frequency spectrum of small {H}elmholtz resonators.
\newblock {\em Proc. A.}, 471(2174):20140339, 18, 2015.

\bibitem{Sommerfeld-1912}
A.~Sommerfeld.
\newblock Die {G}reensche {F}unktion der {S}chwingungsgleichung.
\newblock {\em Jahresbericht der Deutschen Mathematiker-Vereinigung},
  21:309–353, 1912.

\end{thebibliography}

\end{document}